\documentclass[A4paper,11pt]{amsart}

\usepackage[english]{babel}
\usepackage{amsmath}
\usepackage{amssymb} 
\usepackage{amsfonts}
\usepackage[all]{xy}
\usepackage{hyperref}
\usepackage{lscape}

\newtheorem{proposition}{Proposition}[section] 
\newtheorem{lemma}[proposition]{Lemma}
\newtheorem*{lemma*}{Lemma}
\newtheorem{theorem}[proposition]{Theorem}

\theoremstyle{definition}
\newtheorem{definition}[proposition]{Definition}
\newtheorem*{definition*}{Definition}
\newtheorem{example}[proposition]{Example}
\newtheorem*{example*}{Example}
\newtheorem{remark}[proposition]{Remark}

\numberwithin{equation}{section}

\def\ox{\otimes}
\def\M{\mathbb M}

\def\op{\mathsf{op}}
\def\bl{\raisebox{1pt}{\scalebox{.6}{$\,\blacktriangleright\,$}}}
\def\br{\raisebox{1pt}{\scalebox{.6}{$\,\blacktriangleleft\,$}}}
\def\wl{\raisebox{1pt}{\scalebox{.7}{$\,\triangleright\,$}}}
\def\wr{\raisebox{1pt}{\scalebox{.7}{$\,\triangleleft\,$}}}

\begin{document}

\title[Yetter-Drinfeld modules
over weak multiplier bialgebras]{Yetter-Drinfeld modules\\ 
over weak multiplier bialgebras}
\author{Gabriella B\"ohm}
\address{Wigner Research Centre for Physics, H-1525 Budapest 114,
P.O.B.\ 49, Hungary} 
\email{bohm.gabriella@wigner.mta.hu}
\thanks{This research was supported by the Hungarian Scientific Research Fund
OTKA, grant K108384.} 
\begin{abstract}
We continue the study of the representation theory of a {\em regular weak
multiplier bialgebra with full comultiplication}, started in \cite{BoGTLC:wmba,
Bo:wmba_comod}. {\em Yetter-Drinfeld modules} are defined as modules and
comodules, with compatibility conditions that are equivalent to a canonical
object being (weakly) central in the category of modules, and equivalent also
to another canonical object being (weakly) central in the category of
comodules. Yetter-Drinfeld modules are shown to constitute a monoidal category
via the (co)module tensor product over the base (co)algebra. Finite dimensional
Yetter-Drinfeld modules over a regular weak multiplier Hopf algebra with full
comultiplication are shown to possess duals in this monoidal category.
\end{abstract}
\date{Nov 2013}
\maketitle

\section*{Introduction}

{\em Weak bialgebra} (over a field) \cite{Nill, WHAI} is a generalization of
usual bialgebra in the following sense. It is a unital algebra and a counital
coalgebra but the multiplication and the unit are not required to be coalgebra
maps --- symmetrically, the comultiplication and the counit are not required
to be algebra maps. Instead, some weaker axioms are imposed which still
ensure the expected behavior of the representation categories (see
e.g. \cite{BoCaJa}). Namely, any (not necessarily finite dimensional) weak
bialgebra $A$ has a distinguished unital subalgebra and counital quotient
coalgebra $R$ (named the {\em base (co)algebra}). It is isomorphic to the base
field if and only if $A$ is a usual bialgebra. The algebra and coalgebra
structures of $R$ obey compatibility conditions that make it a separable
Frobenius algebra \cite{Scha:wha_qgr,Szl:FinGal}. Consequently, the category
of $R$-modules and the category of $R$-comodules are isomorphic and also the
$R$-module tensor product of any two $R$-(co)modules is isomorphic to their
$R$-comodule tensor product. In fact, any weak bialgebra $A$ contains not only
$R$ but $R\ox R^{\mathsf{op}}$ both as a subalgebra and as a quotient
coalgebra. Therefore any $A$-module and any $A$-comodule carries a natural
$R$-bi(co)module structure. Via the $R$-(co)module tensor product, both the
category $\mathsf M_A$ of $A$-modules and the category $\mathsf M^A$ of
$A$-comodules are monoidal. That is to say, there are strict monoidal
forgetful functors $U_A$ from $\mathsf M_A$, and $U^A$ from $\mathsf M^A$, to
the category ${}_R \mathsf M_R$ of $R$-bi(co)modules. Similarly to the case of
usual (i.e. non-weak) bialgebras studied in \cite{Majid}, the category
$\mathsf{M}^A$ can be recovered from the strict monoidal forgetful functor
$U_A:\mathsf{M}_A\to {}_R \mathsf{M}_R$ as a `commutant' of the range of $U_A$
in ${}_R \mathsf{M}_R$. Symmetrically, the category $\mathsf{M}_A$ can be
re-covered from the strict monoidal forgetful functor $U^A:\mathsf{M}^A\to
{}_R \mathsf{M}_R$ as a `commutant' of the range of $U^A$ in ${}_R
\mathsf{M}_R$. The category of {\em Yetter-Drinfeld modules} over a weak
bialgebra $A$ can be described as the weak monoidal center of $\mathsf{M}_A$,
equivalently, as the weak monoidal center of $\mathsf{M}^A$, see \cite{Nenciu,
Cae:YD}. Thus it is again a monoidal category with a strict monoidal
forgetful functor to ${}_R \mathsf{M}_R$. 

A {\em weak Hopf algebra} is a weak bialgebra equipped with an antipode in the
appropriate sense \cite{WHAI}. If the antipode of a weak Hopf algebra $A$ is
bijective, then finite dimensional Yetter-Drinfeld $A$-modules possess duals
in the monoidal category of Yetter-Drinfeld $A$-modules \cite{Cae:YD}.

{\em Weak multiplier Hopf algebra} was introduced in
\cite{VDaWa:Banach,VDaWa}. It is a generalization of weak Hopf algebra in the
same spirit as Van Daele's multiplier Hopf algebra \cite{VDae:MHA} extends a
usual Hopf algebra. That is, the underlying algebra $A$ does not need to
possess a unit. Instead, surjectivity and non-degeneracy of the multiplication
is required. In this situation, there is a (largest possible) unital algebra
$\M(A)$ (known as the {\em multiplier algebra} of $A$) such that $A$ is a
dense ideal in $\M(A)$. A weak multiplier Hopf algebra is not a coalgebra
either but it has a generalized comultiplication --- which is a multiplicative
map from $A$ to the multiplier algebra of the tensor product algebra $A\ox A$
--- and a generalized counit in the appropriate sense. Among the compatibility
axioms, certain canonical maps $A\ox A \to A\ox A$ are required to be weakly
(i.e. Von Neumann) invertible. The existence and the expected properties of
the antipode map $A\to \M(A)$ follow from the axioms. A weak multiplier Hopf
algebra is said to be {\em regular} if the opposite algebra $A^{\mathsf{op}}$
is a weak multiplier Hopf algebra too. In this case the antipode restricts to
an isomorphism $A\to A$ \cite{VDaWa}. 

{\em Weak multiplier bialgebra} was defined and studied in
\cite{BoGTLC:wmba}. One expects that it should be `weak multiplier Hopf 
algebra without an antipode' and this is almost the case. In fact, any weak
multiplier bialgebra which possesses an antipode is a weak multiplier Hopf
algebra. The converse is not true, however: one of the axioms of weak
multiplier bialgebra (expressing some compatibility between the counit and the
multiplication) may not hold in an arbitrary weak multiplier Hopf algebra
(though it holds in a regular weak multiplier Hopf algebra). 

As in the Hopf case, also a weak multiplier bialgebra $A$ is termed {\em
regular} whenever $A^{\mathsf{op}}$ is a weak multiplier bialgebra
too. Whenever a regular weak multiplier bialgebra $A$ obeys a further
property called {\em fullness of the comultiplication} (see \cite{VDaWa}), the
multiplier algebra $\M(A)$ turns out to have a distinguished (non-unital)
subalgebra $R$. This subalgebra is called the {\em base (co)algebra} for a
classical analogy. Indeed, it was shown in \cite{BoGTLC:wmba} that $R$
possesses the structure of a coseparable coalgebra (that can be considered as
a non-unital generalization of separable Frobenius algebra). As discussed in
\cite{BoVe,BoGT:fF}, this implies that the category of firm $R$-modules and
the category of $R$-comodules are isomorphic and also the $R$-module tensor
product of any two firm $R$-modules (i.e. of $R$-comodules) is isomorphic to
their $R$-comodule tensor product. 

In the study of weak multiplier bialgebras, in the absence of an algebraic
unit, it is a crucial question what is the appropriate notion of {\em
module}. Working with any associative action seems to be too general: we can
not see a monoidal structure of the category of such modules. However,
restricting to modules --- over a regular weak multiplier bialgebra $A$ with a
full comultiplication --- whose action is surjective and non-degenerate, we
can observe the expected behavior: As proven in \cite{BoGTLC:wmba}, any such
module can be equipped with the structure of a firm $R$-bimodule. What is
more, they constitute a monoidal category $\mathsf{M}_{(A)}$ admitting a
strict monoidal forgetful functor $U_{(A)}$ to the category ${}_R
\mathsf{M}_R$ of firm $R$-bimodules. 

Since a weak multiplier bialgebra $A$ has no coalgebra structure, it is not
obvious either what to mean by an $A$-{\em comodule}. This question was
addressed in \cite{Bo:wmba_comod}. As in \cite{VDaZha:corep_I} in the case of
a multiplier Hopf algebra, a right $A$-comodule $M$ is defined by a pair of
compatible linear maps $\lambda,\varrho:M\ox A \to M\ox A$ (generalizing the
left and right multiplication by the image of $m\in M$ under the usual
coaction of a bialgebra $A$ on $M$). Coassociativity of such a generalized
comodule can be formulated as a pentagonal identity, see
\cite{VDaZha:corep_I}. In the case of a {\em weak} multiplier bialgebra, it
has to be supplemented by a normalization condition. The property which
generalizes the counitality of a usual coaction, appears as the {\em
fullness} of $\lambda$ and $\varrho$. As proven in \cite{Bo:wmba_comod}, full
comodules over a regular weak multiplier algebra $A$ with a full
comultiplication carry a firm $R$-bimodule structure. Moreover, they
constitute a monoidal category $\mathsf{M}^{(A)}$ admitting a strict monoidal
forgetful functor $U^{(A)}$ to the category ${}_R \mathsf{M}_R$ of firm
$R$-bimodules.

As a next step, the aim of this paper is to define and study {\em
Yetter-Drinfeld modules} over a regular weak multiplier bialgebra $A$ with full
comultiplication. In Section \ref{sec:embed} we study the `commutants'
$(\mathsf{M}_{(A)})^{U_{(A)}}$ and ${}^{U^{(A)}}(\mathsf{M}^{(A)})$ of the
ranges of the functors $U_{(A)}$ and $U^{(A)}$ in ${}_R
\mathsf{M}_R$. Although --- in contrast to usual (i.e. unital) weak bialgebras
--- they are not isomorphic to $\mathsf{M}^{(A)}$ and $\mathsf{M}_{(A)}$,
respectively, there are full embeddings $\mathsf{M}^{(A)} \to
(\mathsf{M}_{(A)})^{U_{(A)}}$ and $\mathsf{M}_{(A)} \to {}^{U^{(A)}}
(\mathsf{M}^{(A)})$. Applying these embedding functors to a vector space $X$
carrying both a non-degenerate surjective $A$-action and a full $A$-coaction,
we obtain canonical natural transformations $X\ox_R U^{(A)}(-)\to
U^{(A)}(-)\ox_R X$ and $U_{(A)}(-)\ox_R X\to X\ox_R U_{(A)}(-)$. In Section
\ref{sec:YD} we prove that if the induced $R$-actions on $X$ coincide, then
the first one of these natural transformations defines an object of the weak
left center of $\mathsf{M}^{(A)}$ if and only if the second one defines an
object of the weak right center of $\mathsf{M}_{(A)}$. We use these equivalent
properties to define a Yetter-Drinfeld module over $A$. The resulting category
of Yetter-Drinfeld $A$-modules is shown in Section \ref{sec:YD_cat} to be
monoidal with strict monoidal forgetful functors to $\mathsf{M}_{(A)}$ and to
$\mathsf{M}^{(A)}$. If $A$ is a regular weak multiplier Hopf algebra in the
sense of \cite{VDaWa} (that is, a regular weak multiplier bialgebra with a
full comultiplication such that both $A$ and its opposite possess an
antipode), then we show that finite dimensional Yetter-Drinfeld $A$-modules
possess duals in this monoidal category. Some preliminary results about weak
multiplier bialgebras and Hopf algebras, and about their modules and
comodules, are recalled in Section \ref{sec:prelims} from \cite{VDaWa:Banach},
\cite{VDaWa}, \cite{BoGTLC:wmba} and \cite{Bo:wmba_comod}. 

\section{Preliminaries}\label{sec:prelims}

In this section we recall from \cite{VDaWa:Banach}, \cite{VDaWa},
\cite{BoGTLC:wmba} and \cite{Bo:wmba_comod} the most important information
about weak multiplier bialgebras and weak multiplier Hopf algebras, their
modules and comodules. No proofs that can be found in these papers are
repeated. However, we prove here a few technical lemmata for later use. We
also revisit the construction of commutants of strict monoidal functors in
\cite{Majid}.
 
\subsection{Notation}

For any vector space $P$, we denote by $P$ also the identity map $P\to P$. On
elements $p\in P$, also the notation $p\mapsto 1p$ or $p\mapsto p1$ is used
for the action of the identity map, meaning multiplication by the unit $1$ of
the base field. For any vector spaces $P$ and $Q$, we denote by
$\mathsf{Lin}(P,Q)$ the vector space of linear maps $P\to Q$. We denote by
$\mathsf{tw}:P\ox Q \to Q\ox P$ the flip map $p\ox q \mapsto q\ox p$. For a
linear map $f:P\ox Q \to P'\ox Q'$, we use the leg numbering notation
$f^{21}:= \mathsf{tw} f \mathsf{tw}:Q\ox P \to Q'\ox P'$. For any
further vector space $R$, we also use the notation 
$$
\begin{array}{lcl}
f^{13}:=(P' \ox \mathsf{tw})(f\ox R)(P \ox \mathsf{tw})&:
&P \ox R \ox Q \to P' \ox R \ox Q', \\
f^{31}:=(Q' \ox \mathsf{tw})(f^{21}\ox R)(Q \ox \mathsf{tw})&: 
&Q \ox R \ox P \to Q' \ox R \ox P', 
\end{array}
$$
and its variants. For a subset $P_0$ of a vector space $P$, we denote by
$\langle P_0 \rangle$ the subspace of $P$ spanned by the elements of $P_0$.

\subsection{Multiplier algebra}

Consider an algebra $A$ over a field with an associative multiplication
$\mu:A\ox A \to A$ but possibly without a unit. On elements $a,b\in A$,
multiplication is denoted by juxtaposition: $\mu(a\ox b)=ab$. The algebra $A$
is called {\em idempotent} if $\mu$ is surjective; that is, any element of $A$
is a linear combination of elements of the form $ab$ for $a,b\in A$. The
multiplication $\mu$ (or with an alternative terminology, the algebra $A$) is
{\em non-degenerate} if both maps $A\to \mathsf{Lin}(A,A)$,
$$
a\mapsto \mu(a\ox -)\quad \textrm{and} \quad a\mapsto \mu(-\ox a)
$$
are injective. If $A$ is an idempotent and non-degenerate algebra, then it was
proven in \cite{Dauns:Multiplier} that there is a unique largest unital
algebra $\M(A)$ which contains $A$ as a dense ideal (i.e. such that if
$aw=0$ for all $a\in A$ or $w a=0$ for all $a\in A$, then $\M(A)\ni
w=0$). This algebra $\M(A)$ is called the {\em multiplier algebra} of $A$
and it is of the following explicit form. An element $w$ of $\M(A)$ is
given by two linear maps $A\to A$, denoted by 
$$
a\mapsto w a \qquad \textrm{and}\qquad a\mapsto a w,
$$
respectively, such that $a(w b)=(aw) b$ for all $a,b\in A$ (hence
$w(-)$ is a right $A$-module map and $(-)w$ is a left $A$-module
map). The multiplication --- to be denoted by $\mu$ again, and written as
juxtaposition of elements --- is provided by the composition and opposite
composition of these maps and the unit is the pair whose members are equal to
the identity map. The embedding 
$$
A \to \M(A),\qquad 
a\mapsto (\mu(a\ox -),\mu(-\ox a))
$$
makes $A$ a dense ideal, indeed. 

For any idempotent non-degenerate algebra $A$, also the {\em opposite algebra}
obtained from $A$ by reversing the order of the multiplication (that is,
replacing $\mu$ by $\mu^{\mathsf{op}}:=\mu \mathsf{tw}$) is again an
idempotent non-degenerate algebra. Its multiplier algebra obeys
$\M(A^{\mathsf{op}})\cong \M(A)^{\mathsf{op}}$. For idempotent non-degenerate
algebras $A$ and $B$, the tensor product algebra $A\ox B$ is an idempotent
non-degenerate algebra again. For an element $w\in \M(A\ox B)$, we extend our
leg numbering convention by putting for any labels $i,j$,
$$
w^{ij}(-)=[w(-)]^{ij}\qquad\textrm{and}\qquad
(-)w^{ij}=[(-)w]^{ij}.
$$

The following theorem due to Van Daele and Wang is of crucial importance.

\begin{theorem}\cite[Proposition A.3]{VDaWa:Banach}\label{thm:extend}
Let $A$ and $B$ be idempotent non-degenerate algebras. Let $\phi:A\to \M(B)$
be a multiplicative map and $e$ be an idempotent element of $B$ (i.e. such that
$e^2=e$). If 
$$
\langle \phi(a)b\ |\ a\in A,b\in B\rangle=\langle eb\ | \ b\in B\rangle
\quad \textrm{and}\quad
\langle b\phi(a)\ |\ a\in A,b\in B\rangle=\langle be\ | \ b\in B\rangle
$$
then there is a unique multiplicative map $\overline \phi:\M(A)\to \M(B)$ such
that $\overline \phi(1)=e$ and for all elements $a\in A$, $\overline
\phi(a)=\phi(a)$.
\end{theorem}

\subsection{Weak multiplier bialgebra}
A {\em weak multiplier bialgebra} $A$ over a field $k$ is given by the
following data. 
\begin{itemize}
\item an idempotent non-degenerate $k$-algebra $A$,
\item linear maps $E_1,E_2,T_1,T_2:A\ox A \to A\ox A$,
\item a linear map $\epsilon:A\to k$.
\end{itemize}
They are subject to the axioms (i)-(ix) in \cite[Definition
1.1]{Bo:wmba_comod}. An equivalent but more traditional definition can be
found in \cite[Definition 2.1]{BoGTLC:wmba}. 

By axioms (iii) and (iv) in \cite[Definition 1.1]{Bo:wmba_comod}, the maps
$T_1$ and $T_2$ determine a multiplicative map $\Delta:A \to \M(A\ox
A)$ --- called the {\em comultiplication} --- via the prescriptions 
$$
\Delta(a)(b\ox c)=T_1(a\ox c)(b\ox 1)\quad \textrm{and}\quad
(b\ox c)\Delta(a)=(1\ox c)T_2(b\ox a),\qquad \forall a,b,c\in A.
$$
By axioms (i) and (ii) in \cite[Definition 1.1]{Bo:wmba_comod}, the maps $E_1$
and $E_2$ determine an idempotent element $E\in \M(A\ox A)$ via 
$$
E(a\ox b)=E_1(a\ox b)\qquad \textrm{and}\qquad
(a\ox b)E=E_2(a\ox b)\qquad \forall a,b\in A.
$$
By axiom (vii) in \cite[Definition 1.1]{Bo:wmba_comod} and by Theorem
\ref{thm:extend}, $\Delta$ extends to a multiplicative map $\overline
\Delta:\M(A)\to \M(A\ox A)$ such that $\overline \Delta(1)=E$.

By \cite[Propositions 2.4 and 2.6]{BoGTLC:wmba}, the idempotent element $E$
allows for the definition of linear maps 
\begin{eqnarray}
\label{eq:PiBarR}
&\overline \sqcap^R: A \to \M(A),\qquad
&a\mapsto (\M(A)\ox \epsilon)[E(1\ox a)]\\
&\overline \sqcap^L: A \to \M(A),\qquad
&a\mapsto (\epsilon \ox \M(A))[(a\ox 1)E].
\label{eq:PiBarL}
\end{eqnarray}
Their properties are studied in Section 3 of \cite{BoGTLC:wmba}.

A weak multiplier bialgebra $A$ is said to be {\em regular} if there is
another weak multiplier bialgebra given by 
\begin{itemize}
\item the idempotent non-degenerate $k$-algebra $A^{\mathsf{op}}$,
\item the same linear maps $E_2,E_1:A\ox A \to A\ox A$ (but their roles
interchanged) and the linear maps 
\begin{eqnarray*}
&T_3:A\ox A \to A\ox A \qquad &a\ox b \mapsto (1\ox b)\Delta(a)\\
&T_4:A\ox A \to A\ox A \qquad &a\ox b \mapsto \Delta(b)(a\ox 1),
\end{eqnarray*}
\item the same linear map $\epsilon:A\to k$.
\end{itemize}
This is equivalent to the existence of {\em some} linear maps $T_3$ and $T_4$
obeying axiom (x) in \cite[Definition 1.1]{Bo:wmba_comod} (which axiom implies
that the induced comultiplication on $A^{\mathsf{op}}$ is equal to the
comultiplication $\Delta$ on $A$). For a regular weak multiplier bialgebra
$A$, there are further two linear maps 
\begin{eqnarray}
\label{eq:PiR}
&\sqcap^R: A \to \M(A),\qquad 
&a\mapsto (\M(A)\ox \epsilon)[(1\ox a)E]\\
\label{eq:PiL}
&\sqcap^L: A \to \M(A),\qquad
&a\mapsto (\epsilon \ox \M(A))[E(a\ox 1)].
\end{eqnarray}

The comultiplication of a weak multiplier bialgebra is said to be {\em right
full} if the subspace 
$$
\langle(A\ox \omega)T_1(a\ox b)\ |\ a,b\in A,\ \omega \in \mathsf{Lin}(A,k)
\rangle
$$
is equal to $A$. By \cite[Theorem 3.13]{BoGTLC:wmba}, the comultiplication of
a regular weak multiplier bialgebra is right full if and only if the ranges of
the maps $\overline \sqcap^R$ and $\sqcap^R$ coincide. In this situation this
coinciding range will be denoted by $R$. Symmetrically, the comultiplication
of a weak multiplier bialgebra is {\em left full} if the subspace 
$$
\langle
(\omega\ox A)T_2(a\ox b)\ |\ a,b\in A,\ \omega \in \mathsf{Lin}(A,k)
\rangle
$$
is equal to $A$. By \cite[Theorem 3.13]{BoGTLC:wmba}, the comultiplication of
a regular weak multiplier bialgebra is left full if and only if the ranges of
the maps $\overline \sqcap^L$ and $\sqcap^L$ coincide. In this situation this
coinciding range will be denoted by $L$. 

Consider a regular weak multiplier bialgebra $A$ over a field $k$ with a 
right full comultiplication. Then the subspace $R$ of $\M(A)$ above carries
the following structure.
\begin{itemize}
\item By \cite[Lemma 3.4]{BoGTLC:wmba}, $R$ is a (non-unital) subalgebra of
 $\M(A)$.
\item By \cite[Theorem 4.4]{BoGTLC:wmba}, $R$ is a coalgebra via the
 comultiplication and counit
$$
\delta:R \to R\ox R,\ \ \sqcap^R(ab)\mapsto (\sqcap^R\ox \sqcap^R)T_2(a\ox b)
\qquad 
\varepsilon:R\to k, \ \ \sqcap^R(a)\mapsto \epsilon(a).
$$
\item By \cite[Proposition 4.3~(3)]{BoGTLC:wmba}, the multiplication in $R$
 is split by the comultiplication $\delta$. 
\item By \cite[Proposition 4.3~(3)]{BoGTLC:wmba}, the multiplication in $R$ is
 a morphism of left and right $R$-comodules. Equivalently, the
 comultiplication $\delta$ is a morphism of left and right $R$-modules. 
\item By \cite[Theorem 4.6~(2)]{BoGTLC:wmba}, the algebra $R$ has (idempotent)
 local units (so in particular it is a firm\footnote{A (say, right)
 module $P$ over a non-unital algebra $R$ is said to be {\em firm} if the
 $R$-action $P\ox R \to P$ projects to a vector space isomorphism $P\ox_R R\to
 P$. The algebra $R$ is {\em firm} if it is a firm left, equivalently, right
 $R$-module via the multiplication. This terminology is attributed to an
 unpublished preprint by D. Quillen in 1997.} algebra). 
\end{itemize}
All these amount to saying that $R$ is a {\em coseparable coalgebra} (and hence
a {\em firm Frobenius algebra}, see \cite{BoGT:fF}). Then we know from
\cite[Proposition 2.17]{BoVe} that the category of firm modules over the firm
algebra $R$ is isomorphic to the category of comodules over the coalgebra $R$;
and the $R$-module tensor product $V\ox_R W$ of any firm $R$-modules $V$ and
$W$ is isomorphic to their $R$-comodule tensor product. This implies that the
canonical epimorphism 
$$
\pi_{V,W}:V\ox W \xymatrix@C=15pt{\ar@{->>}[r]&} V\ox_R W
$$
is split by the map
$$
\iota_{V,W}:V\ox_R W \xymatrix@C=10pt{\,\ar@{>->}[r]&} V \ox W,\qquad 
v\cdot r\ox_R w\mapsto (v\cdot(-)\ox (-)\cdot w)\delta(r).
$$
For $R$-bimodules $V,W$ and $Z$, we denote by $\pi_{V,W,Z}$ the epimorphism 
$$
\pi_{V\ox_R W,Z}(\pi_{V,W}\ox Z)=
\pi_{V,W\ox_R Z}(V\ox \pi_{W,Z}):
V\ox W \ox Z\to V\ox_R W \ox_R Z
$$
and we denote by $\iota_{V,W,Z}$ its section
$$
(\iota_{V,W}\ox Z)\iota_{V\ox_R W,Z}=
(V\ox \iota_{W,Z})\iota_{V,W\ox_R Z}:
V\ox_R W \ox_R Z\to V\ox W \ox Z.
$$
If the comultiplication of a regular weak multiplier bialgebra $A$ is left
full, then the subspace $L$ has analogous structures. 

If the comultiplication is right full, then by \cite[Lemma 4.8]{BoGTLC:wmba}
there are anti-multiplicative linear maps
\begin{eqnarray}
\label{eq:tau}
&\tau: R \to \overline \sqcap^L(A),\qquad &
\sqcap^R(a)\mapsto \overline \sqcap^L(a)\\
\label{eq:taubar}
&\overline \tau: R\to \sqcap^L(A),\qquad &
\overline \sqcap^R(a)\mapsto \sqcap^L(a).
\end{eqnarray}
If the comultiplication is left full, then by \cite[Lemma 4.8]{BoGTLC:wmba}
there are anti-multiplicative linear maps
\begin{eqnarray}
\label{eq:sigma}
&\sigma: L\to \sqcap^R(A),\qquad &\overline \sqcap^L(a)\mapsto \sqcap^R(a)\\
\label{eq:sigmabar}
&\overline \sigma: L\to \overline \sqcap^R(A),\qquad 
&\sqcap^L(a)\mapsto \overline \sqcap^R(a).
\end{eqnarray}
If the comultiplication is both left and right full, then by \cite[Proposition
4.9]{BoGTLC:wmba} these are anti-coalgebra isomorphisms such that
$\tau=\sigma^{-1}$ and $ \overline\tau= \overline\sigma^{\,-1}$.
The composite map $\vartheta:=\sigma \overline\sigma\, ^{-1}$ is the {\em
Nakayama automorphism} of $R$ --- meaning 
$\varepsilon(rs)=\varepsilon(\vartheta(s)r)$ for all $r,s\in R$.
The map $\overline\sigma\, ^{-1}\sigma$ is the Nakayama automorphism of $L$.

By \cite[Proposition 4.3~(1) and Proposition 4.11]{BoGTLC:wmba}, the formulae
\begin{equation}\label{eq:F}
F(a\ox b):=((A\ox \sigma)[E(a\ox 1)])(1\ox b)\qquad
(a\ox b)F:=(1\ox b)((A\ox \sigma)[(a\ox 1)E])
\end{equation}
define a multiplier $F$ on $A\ox A$. By a symmetric form of \cite[Proposition
4.11]{BoGTLC:wmba}, 
\begin{equation}\label{eq:F^op}
F^{21}(a\ox b):=((A\ox \overline\sigma)[E(a\ox 1)])(1\ox b)\qquad
(a\ox b)F^{21}:=(1\ox b)((A\ox \overline\sigma)[(a\ox 1)E]).
\end{equation}

\begin{lemma}\label{lem:dual_YD}
For a regular weak multiplier bialgebra $A$ with left and right full
comultiplication, and for any $a\in A$, the following expressions are equal.
\begin{itemize}
\item[{(a)}] $(R\ox \sqcap^R)[E(1\ox a)]$,
\item[{(b)}] $(\overline \sqcap^R \ox R)[(a\ox 1)F]$,
\item[{(c)}] $(\overline \sqcap^R \ox R)[(a\ox 1)E^{21}]$,
\item[{(d)}] $(R\ox \sqcap^R)[F(1\ox a)]$.
\end{itemize}
\end{lemma}

\begin{proof}
(a)=(b) By \cite[Proposition 2.5~(1)]{BoGTLC:wmba}, $(a\ox 1)E\in A\ox L$ and
by (2.3) in \cite{BoGTLC:wmba}, $E(1\ox a)\in R\ox A$. They obey 
$$ 
(\overline \sqcap^R\ox L)[(a\ox 1)E]\stackrel{\eqref{eq:PiBarR}}=
(R\ox \epsilon \ox L)[(E\ox 1)(1\ox a\ox 1)(1\ox E)]
\stackrel{\eqref{eq:PiBarL}}=
(R\ox \overline \sqcap^L)[E(1\ox a)].
$$
Applying $R\ox \sigma$ (cf. \eqref{eq:sigma}) to both sides and using
\eqref{eq:F}, we obtain (b)=(a). 

(c)=(d) follows symmetrically using \eqref{eq:PiR} and \eqref{eq:PiL} together
with \eqref{eq:F^op}.

(a)=(c) Using axiom (iii) in \cite[Definition 1.1]{Bo:wmba_comod} in the
second equality, (3.6) in \cite{BoGTLC:wmba} in the third equality and
\cite[Proposition 4.3~(1)]{BoGTLC:wmba} in the penultimate one, it follows for
any $a,b,c\in A$ that 
\begin{eqnarray*}
(c\ox 1)((A\ox \sqcap^R)T_1(a\ox b))&=&
(A\ox \sqcap^R)[(c\ox 1)T_1(a\ox b)]\\
&=&
(A\ox \sqcap^R)[T_2(c\ox a)(1\ox b)]\\
&=&
(A\ox \sqcap^R)[((A\ox \sqcap^R)T_2(c\ox a))(1\ox b)]\\
&=&
(A\ox \sqcap^R)[(ca\ox 1)F(1\ox b)]\\
&=&
(c\ox 1)((A\ox \sqcap^R)[(a\ox 1)F(1\ox b)]).
\end{eqnarray*}
Thus using the non-degeneracy of $A$ and simplifying by $c$, $(A\ox
\sqcap^R)T_1(a\ox b)=(A\ox \sqcap^R)[(a\ox 1)F(1\ox b)]$. 
Symmetrically, using (1.6) in \cite{Bo:wmba_comod}, \cite[Lemma
3.2]{BoGTLC:wmba} and \cite[Proposition 4.3~(1)]{BoGTLC:wmba}, it follows
that $(\overline \sqcap^R\ox A)T_3^{21}(a\ox b)=(\overline \sqcap^R\ox
A)[(a\ox 1)F(1\ox b)]$. With these identities at hand,
\begin{eqnarray*}
(R\ox \sqcap^R)[E(1\ox ab)]&=&
(\overline \sqcap^R\ox \sqcap^R)T_1(a\ox b)=
(\overline \sqcap^R\ox \sqcap^R)[(a\ox 1)F(1\ox b)]\\
&=&(\overline \sqcap^R\ox \sqcap^R)T_3^{21}(a\ox b)=
(\overline \sqcap^R \ox R)[(ab\ox 1)E^{21}],
\end{eqnarray*}
from which we conclude by the idempotency of $A$. In the first equality we
used (2.3) in \cite{BoGTLC:wmba} and in the last one we used (3.4) in
\cite{BoGTLC:wmba}. 
\end{proof}

\subsection{Antipode}
The {\em antipode} on a regular weak multiplier bialgebra $A$ is a linear map
$S:A\to \M(A)$ satisfying the conditions in part (2) of \cite[Theorem
6.8]{BoGTLC:wmba}. Then it follows that
\begin{equation}\label{eq:6.14}
\begin{array}{ll}
\mu(S\ox A)T_1=\mu(\sqcap^R \ox A)\qquad 
&\mu(A\ox S)T_2=\mu(A\ox \sqcap^L)\\
\mu(S\ox A)E_1=\mu(S\ox A)\qquad 
&\mu(A\ox S)E_2=\mu(A\ox S),
\end{array}
\end{equation}
see (6.14) in \cite{BoGTLC:wmba}. If the antipode exists then it is
unique. Any weak multiplier bialgebra possessing an antipode is a weak
multiplier Hopf algebra in the sense of \cite{VDaWa:Banach, VDaWa} but not
conversely.

By \cite[Theorem 6.12]{BoGTLC:wmba} $S$ is anti-multiplicative. That is, for
any $a,b\in A$, $S(ab)=S(b)S(a)$. By \cite[Proposition 6.13]{BoGTLC:wmba} $S$
is non degenerate. That is, any element of $A$ can be written as a linear
combination of elements of the form $aS(b)$ and it can be written also as a
linear combination of elements of the form $S(b)a$, in terms of elements
$a,b\in A$. These assertions together with Theorem \ref{thm:extend} imply that
$S$ extends to a unital anti-algebra map $\overline S:\M(A)\to \M(A)$. 
Assuming also that the comultiplication of $A$ is left and right full, it
follows by \cite[Corollary 6.16]{BoGTLC:wmba} that $S$ is
anti-comultiplicative. That is, in terms of the opposite comultiplication
$\Delta^{\mathsf{op}}:A\to \M(A\ox A)$ --- defined by
$\Delta^{\mathsf{op}}(a)=\Delta(a)^{21}$ for any $a\in A$ --- the equality
$\overline \Delta S =\overline{(S\ox S)}\Delta^{\mathsf{op}}$ holds. By
\cite[Lemma 6.14]{BoGTLC:wmba}, the restriction of $\overline S$ to $R$ is
equal to $\overline \sigma\, ^{-1}$ and the restriction of $\overline S$ to
$L$ is equal to $\sigma$.

A {\em regular weak multiplier Hopf algebra} in \cite{VDaWa:Banach, VDaWa} is
the same as a regular weak multiplier bialgebra $A$ with left and right full
comultiplication such that both $A$ and its opposite $A^{\mathsf{op}}$ possess
an antipode. The following is a variant of
\cite[Theorem 4.10]{VDaWa}:

\begin{theorem}\label{thm:reg_wmha}
For a regular weak multiplier bialgebra $A$ with left and right full
comultiplication possessing an antipode $S$, the following are equivalent.
\begin{itemize}
\item[{(a)}] $A$ is a regular weak multiplier Hopf algebra.
\item[{(b)}] $S$ factorizes through a vector space isomorphism $A\to A$ (to be
 denoted by $S$ too) via the embedding $A\to \M(A)$.
\end{itemize}
\end{theorem}

\noindent
If assertion (b) in Theorem \ref{thm:reg_wmha} holds then $S^{-1}$ defines the
antipode of $A^{\mathsf{op}}$.

\subsection{Modules}\label{sec:module}\ 
Since there is an associative algebra underlying any weak multiplier
bialgebra $A$, there is an evident notion of right $A$-{\em module} $V$ with
an associative action 
$\cdot:V\ox A \to V$: 
$$
v\cdot(ab)=(v\cdot a)\cdot b\qquad \forall v\in V,\ a,b\in A.
$$
A morphism of $A$-modules is a linear map which commutes with the
$A$-actions. 
Similarly to the properties of the multiplication, we say that $V$ is an
{\em idempotent} module if $\cdot:V\ox A \to V$ is surjective (i.e. any
element of $V$ is a linear combination of elements of the form $v\cdot a$, in
terms of $v\in V$ and $a\in A$) and it is {\em non-degenerate} if the map
$$
V\to \mathsf{Lin}(A,V),\qquad
v\mapsto v\cdot (-)
$$
is injective. We consider the category $\mathsf{M}_{(A)}$ whose objects are
the idempotent non-degenerate right $A$-modules and whose morphisms are the
$A$-module maps. 

If $A$ is a regular weak multiplier bialgebra with right full
comultiplication, then any idempotent non-degenerate right $A$-module $V$ was
equipped in \cite[Proposition 5.2]{BoGTLC:wmba} with the structure of an
$R$-bimodule with firm right and left actions:
\begin{equation}\label{eq:w_actions}
(v\cdot a)\wr r:=v\cdot(ar)\qquad
r\wl (v\cdot a):=v\cdot(a \tau(r))\qquad
\forall v\in V, a\in A,\ r\in R,
\end{equation}
where the map \eqref{eq:tau} is used. This defines a functor $U_{(A)}$ ---
acting on the morphisms as the identity map --- from $\mathsf{M}_{(A)}$ to the
category ${}_R \mathsf{M}_R$ of firm $R$-bimodules. Note that ${}_R
\mathsf{M}_R$ is a monoidal category via the $R$-module tensor product as the
monoidal product and the regular $R$-bimodule as the neutral object.
By \cite[Theorem 5.6]{BoGTLC:wmba}, $\mathsf{M}_{(A)}$ admits a monoidal
structure such that $U_{(A)}$ is strict monoidal. 

\subsection{Comodules}\label{sec:comodule}
Applying the philosophy used in \cite{VDaZha:corep_I} for (non-weak)
multiplier Hopf algebras, a {\em comodule} over a regular weak multiplier
bialgebra $A$ is a vector space $M$ equipped with linear maps
$\lambda,\varrho:M\ox A \to M\ox A$ such that 
\begin{itemize}
\item $(1\ox a)\lambda(m\ox b)=\varrho(m\ox a)(1\ox b)$ for all $m\in M$,
 $a,b\in A$, 
\item the same coassociativity condition
$(\lambda\ox A)\lambda^{13}(M\ox T_1)=(M\ox T_1)(\lambda\ox A)$
in \cite{VDaZha:corep_I} holds (for a few more equivalent forms see
\cite[Proposition 2.1]{Bo:wmba_comod}), 
\item together with the normalization
condition 
$(\lambda\ox A)\lambda^{13}(M\ox E_1)=(\lambda\ox A)\lambda^{13}$
(some more equivalent forms can be found again in \cite[Proposition
2.1]{Bo:wmba_comod}). 
\end{itemize}
A morphism of $A$-comodules is a linear map $f:M\to M'$ such that $(f\ox
A)\lambda= \lambda'(f\ox A)$, equivalently, $(f\ox A)\varrho= \varrho'(f\ox
A)$. The triple $(M,\lambda,\varrho)$ is a right comodule over a regular weak
multiplier bialgebra $A$ if and only if $(M,\varrho,\lambda)$ is a right
comodule over the opposite regular weak multiplier bialgebra $A^{\mathsf{op}}$.
Similarly to the comultiplication, a comodule $(M,\lambda,\varrho)$ is said to
be {\em full} if any (hence by \cite[Lemma 4.1]{Bo:wmba_comod} both) of the
subspaces 
\begin{eqnarray*}
&&\langle (M\ox \omega)\lambda(m\ox a)\ |\ 
m\in M,\ a\in A,\ \omega \in \mathsf{Lin}(A,k) \rangle
\quad \textrm{and}\\
&&\langle (M\ox \omega)\varrho(m\ox a)\ |\ 
m\in M,\ a\in A,\ \omega \in \mathsf{Lin}(A,k) \rangle,
\end{eqnarray*}
is equal to $M$. We consider the category $\mathsf{M}^{(A)}$ whose objects are
the full right $A$-comodules and whose morphisms are the $A$-comodule maps. 

If $A$ is a regular weak multiplier bialgebra with a right full
comultiplication, then any full right $A$-comodule $M$ was equipped in
\cite[Theorem 4.5]{Bo:wmba_comod} with the structure of an $R$-bimodule with
firm right and left actions: 
\begin{equation}\label{eq:b_actions}
m\br \overline \sqcap^R(a):=(M\ox \epsilon)\lambda(m\ox a)\qquad
\sqcap^R(a)\bl m:=(M\ox \epsilon)\varrho(m\ox a).
\end{equation}
This defines a functor $U^{(A)}$ --- acting on the morphisms as the identity
map --- from $\mathsf{M}^{(A)}$ to the category ${}_R \mathsf{M}_R$ of firm
$R$-bimodules. By \cite[Theorem 5.7]{Bo:wmba_comod}, $\mathsf{M}^{(A)}$ admits
a monoidal structure such that $U^{(A)}$ is strict monoidal. 

If $A$ is a regular weak multiplier bialgebra with a left and right full
comultiplication and it possesses an antipode, then it was proven in
\cite[Theorem 6.7]{Bo:wmba_comod} that any finite dimensional object in
$\mathsf{M}^{(A)}$ possesses a dual.

\begin{lemma}\label{lem:comod-mod_idempotent}
Let $A$ be a regular weak multiplier bialgebra with right full
comultiplication. Let $V$ be an idempotent non-degenerate right $A$-module
regarded as a firm $R$-bimodule as in \eqref{eq:w_actions}. Let $M$ be a full
right $A$-comodule regarded as a firm $R$-bimodule as in 
\eqref{eq:b_actions}.
\begin{itemize}
\item[{(1)}] The idempotent map $\iota_{M,V}\pi_{M,V}:M\ox V\to M\ox V$ satisfies
$$ 
\iota_{M,V}\pi_{M,V}(m\ox v\cdot a)=
(m\br(-)\ox v\cdot (-))((1\ox a)E),\qquad \forall m\in M,\ v\in V,\ a\in A.
$$
\item[{(2)}] The idempotent map $\iota_{V,M}\pi_{V,M}:V\ox M\to V\ox M$ satisfies
$$ 
\iota_{V,M}\pi_{V,M}(v\cdot a \ox m)=
(v\cdot (-)\ox (-)\bl m)((a\ox 1)F),\qquad \forall m\in M,\ v\in V,\ a\in A.
$$
\end{itemize}
\end{lemma}

\begin{proof}
(1) For any $a\in A$, we have $(1\ox a)E\in R\ox A$ by (3.4) in
\cite{BoGTLC:wmba}. So the right hand side of the equality in the claim is 
meaningful. Since $A$ is idempotent and $M$ is a firm right $R$-module, $M$ is
spanned by elements of the form $m\br \sqcap^R(cd)$, for $m\in M$ and
$c,d\in A$. For such elements,
\begin{eqnarray*}
\iota_{M,V}\pi_{M,V}(m\br \sqcap^R(cd)\ox v\cdot a)&=&
(m\br (-) \ox (-)\wl (v\cdot a))[(\sqcap^R \ox \sqcap^R)T_2(c\ox d)]\\
&=&(m\br (-) \ox v\cdot (-))
[(1\ox a)((\sqcap^R \ox \overline \sqcap^L)T_2(c\ox d))]\\
&=&(m\br (-) \ox v\cdot (-))[(\sqcap^R \ox A)(cd\ox a)E]\\
&=&(m\br \sqcap^R(cd)(-) \ox v\cdot (-))[(1\ox a)E].
\end{eqnarray*}
The first equality follows by (4.1) in \cite{BoGTLC:wmba}, the penultimate one
follows by \cite[Proposition 2.5~(1)]{BoGTLC:wmba}, and the last one does by
\cite[Lemma 1.3~(3)]{Bo:wmba_comod}.

(2) For any $a\in A$, we have $(a\ox 1)F\in A\ox R$ by \cite[Proposition
4.3~(1)]{BoGTLC:wmba}. So the right hand side of the equality in the claim
is meaningful. By \cite[Proposition 4.3~(3)]{BoGTLC:wmba}, $\delta(r)=(r\ox
1)F$ for any $r\in R$. Hence
$$
\iota_{V,M}\pi_{V,M}(v\cdot ar \ox m)=
((v\cdot a)\wr (-) \ox (-)\bl m)\delta(r)=
(v\cdot (-)\ox (-)\bl m)((ar\ox 1)F).
$$
By surjectivity of the right $R$-action on $A$ (cf. \cite[Lemma
3.7~(4)]{BoGTLC:wmba}) this proves the claim.
\end{proof}

\subsection{Commutants of strict monoidal functors}\label{sec:commutant}
The following notion due to Majid in \cite{Majid} generalizes the weak center
of a monoidal category in \cite{Scha:DuDoQgp}, which generalizes the center
construction in \cite{JoyStr}. 

\begin{definition}
For monoidal categories $(\mathsf{M},\ox, I)$ and $(\mathsf{M}',\ox',I')$, and
a strict monoidal functor $U:\mathsf{M}\to \mathsf{M}'$, the {\em right
commutant} of $U$ is the following category $\mathsf{M}^U$: The objects are
pairs consisting of an object $P$ of $\mathsf{M}'$ and a natural transformation
$\sigma: U(-)\ox' P \to P\ox' U(-)$ which is compatible with the monoidal
structures in the sense of the following commutative diagrams. 
$$
\xymatrix@R=15pt@C=28pt{
P\ar@{=}[r]\ar[dd]_-\cong&
P\ar[dd]^-\cong
&
UV\ox'(UW\ox' P)\ar[r]^-{UV\ox' \sigma_W}\ar[dd]_-\cong&
UV\ox'(P\ox' UW)\ar[d]^-\cong\\
&&&
(UV\ox'P)\ox' UW\ar[d]^-{\sigma_V \ox' UW}\\
I'\ox' P\ar@{=}[dd]&
P\ox' I'\ar@{=}[dd]
&
(UV\ox' UW)\ox' P\ar@{=}[dd]&
(P\ox' UV)\ox' UW\ar[d]^-\cong\\
&&&
P\ox' (UV\ox' UW)\ar@{=}[d]\\
UI\ox' P\ar[r]_-{\sigma_I}&
P\ox' UI
&
U(V\ox W)\ox' P\ar[r]_-{\sigma_{V\ox W}}&
P\ox' U(V\ox W)}
$$
Here the arrows carrying the label $\cong$ stand for Mac Lane's coherence
isomorphisms (not to be explicitly denoted later on). The morphisms
$(P,\sigma)\to (P',\sigma')$ are morphisms $f:P\to P'$ in $\mathsf{M}'$
rendering commutative
$$
\xymatrix{
U(-)\ox' P\ar[r]^-\sigma\ar[d]_-{U(-)\ox' f}&
P\ox' U(-)\ar[d]^-{f\ox' U(-)}\\
U(-)\ox' P'\ar[r]_-{\sigma'}&
P'\ox' U(-).}
$$
The {\em left commutant} ${}^U \mathsf{M}$ of $U$ is defined symmetrically. 
Its objects are pairs consisting of an object $Q$ of $\mathsf{M}'$ and a
natural transformation $\vartheta: Q\ox' U(-)\to U(-)\ox' Q$ satisfying the
evident compatibility conditions with the monoidal structures. The morphisms
are morphisms in $\mathsf{M}'$ which commute with the structure morphisms
$\vartheta$. 
\end{definition}

\begin{theorem}\cite[Theorem 3.3]{Majid}
For any monoidal categories $(M,\ox, I)$ and $(M',\ox',I')$, and a strict
monoidal functor $U:M\to M'$, the left and right commutants of $U$ are
monoidal categories admitting a strict monoidal forgetful functor to $M'$.
\end{theorem}

For later reference, we recall that the monoidal product of two objects
$(P,\sigma)$ and $(P',\sigma')$ in $\mathsf{M}^U$ is 
$$
(P\ox' P', 
\xymatrix{
U(-)\ox' P\ox' P'
\ar[r]^-{\sigma\ox' P'}& 
P\ox' U(-)\ox' P'
\ar[r]^-{P\ox' \sigma'}& 
P\ox' P'\ox' U(-)})
$$
where we omitted the associativity constraints. The monoidal unit is $R'$ via
the natural transformation built up from the unitors in $\mathsf{M}'$.

The (left or right) commutant of the identity functor on a monoidal category
$\mathsf{M}$ is known as the (left or right) {\em weak center}
\cite{Scha:DuDoQgp}. The left weak center is denoted by
${}^{\mathsf{M}}\mathsf{M}$ and the right weak center is denoted by
$\mathsf{M}^{\mathsf{M}}$. Any strict monoidal functor $U:\mathsf{M}\to
\mathsf{M}'$ lifts to strict monoidal forgetful functors
$\mathsf{M}^{\mathsf{M}}\to \mathsf{M}^U$ and ${}^{\mathsf{M}}\mathsf{M}\to 
{}^U\mathsf{M}$; to be denoted by $U$ again. The {\em center} of a monoidal
category \cite{JoyStr} is the full subcategory of the left weak center for
whose objects $(Q,\vartheta)$ the $\vartheta$ is a natural isomorphism. 

\section{Some embeddings of the module and comodule categories}\label{sec:embed}

If $A$ is a usual weak bialgebra possessing a unit as in \cite{Nill,WHAI},
then both the category of (say, right) $A$-modules and the category of (say,
right) $A$-comodules are monoidal, admitting strict monoidal forgetful
functors to the bimodule category of the so-called base subalgebra (see
e.g. \cite{BoCaJa}). So one can look at the commutants (in the sense recalled
in Section \ref{sec:commutant}) of these forgetful functors. Generalizing the
results in \cite{Majid}, one can see that the left commutant of the forgetful
functor from the category of right $A$-comodules is isomorphic to the category
of right $A$-modules. Symmetrically, the right commutant of the forgetful
functor from the category of right $A$-modules is isomorphic to the category
of right $A$-comodules. 

Such isomorphisms do not seem to hold for a weak multiplier bialgebra
$A$. Instead, the aim of this section is to show that if $A$ is regular and
its comultiplication is right full, then there is a fully faithful functor
from the category $\mathsf{M}_{(A)}$ of idempotent non-degenerate right
$A$-modules in Section \ref{sec:module} to the left commutant of the functor
$U^{(A)}$ in Section \ref{sec:comodule} from the category $\mathsf{M}^{(A)}$
of full right $A$-comodules. Symmetrically, there is a fully faithful functor
from the category $\mathsf{M}^{(A)}$ of full right $A$-comodules in Section 
\ref{sec:comodule} to the right commutant of the functor $U_{(A)}$ in Section
\ref{sec:module} from the category $\mathsf{M}_{(A)}$ of idempotent
non-degenerate right $A$-modules. 

\begin{example}\label{ex:mod}
Let $A$ be a regular weak multiplier bialgebra with right full
comultiplication. Then by \cite[Theorem 5.6]{BoGTLC:wmba}, the idempotent
non-degenerate right $A$-modules constitute a monoidal category
$\mathsf{M}_{(A)}$ admitting a strict monoidal functor $U_{(A)}$ to the
category ${}_R \mathsf{M}_R$ of firm bimodules over the base algebra
$R:=\sqcap^R(A)=\overline \sqcap^R(A)$. The right commutant of $U_{(A)}$ is
the following category $(\mathsf{M}_{(A)})^{U_{(A)}}$. An object is given by a
firm $R$-bimodule $P$ and an $R$-bimodule map $\sigma_V:V\ox_R P \to P\ox_R V$
for any idempotent non-degenerate right $A$-module $V$, such that the
following hold. 
\begin{itemize}
\item Compatibility with the monoidal unit; that is,\\
$\sigma_R(r\ox_R p\cdot s)=r\cdot p \ox_R s$
for all $p\in P$, $r,s\in R$.
\item Compatibility with the monoidal product; that is,\\
$\sigma_{V\ox_R W}=(\sigma_V \ox_R W)(V\ox_R \sigma_W)$,
for all $V,W\in \mathsf{M}_{(A)}$.
\item Naturality; that is,\\
$\sigma_{V'}(g\ox_R P)=(P\ox_R g)\sigma_V$,
for all $g\in \mathsf{Hom}_{(A)}(V,V')$.
\end{itemize} 
A morphism $(P,\sigma)\to (P',\sigma')$ is an $R$-bimodule map $u:P\to P'$
such that 
$$
\sigma'_V(V\ox_R u)=(u\ox_R V)\sigma_V,\qquad
\textrm{for all}\ V\in \mathsf{M}_{(A)}.
$$
\end{example}

\begin{example}\label{ex:comod}
Let $A$ be a regular weak multiplier bialgebra with right full
comultiplication. Then by \cite[Theorem 5.7]{Bo:wmba_comod}, the full right
$A$-comodules constitute a monoidal category $\mathsf{M}^{(A)}$ admitting a
strict monoidal functor $U^{(A)}$ to the category ${}_R \mathsf{M}_R$ of firm
bimodules over the base algebra $R:=\sqcap^R(A)=\overline \sqcap^R(A)$. The
left commutant of $U^{(A)}$ is the following category ${}^{U^{(A)}}
(\mathsf{M}^{(A)})$. An object is given by a firm $R$-bimodule $Q$ and an
$R$-bimodule map $\vartheta_M:Q\ox_R M \to M\ox_R Q$ for any full right
$A$-comodule $M$, such that the following hold. 
\begin{itemize}
\item Compatibility with the monoidal unit; that is,\\
$\vartheta_R(s\cdot q \ox_R r)=s \ox_R q\cdot r$,
for all $q\in Q$, $r,s\in R$.
\item Compatibility with the monoidal product; that is,\\
$\vartheta_{M\ox_R N}=(M\ox_R \vartheta_N)(\vartheta_M \ox_R N)$,
for all $M,N\in \mathsf{M}^{(A)}$.
\item Naturality; that is,\\
$\vartheta_{M'}(Q\ox_R f)=(f\ox_R Q)\vartheta_M$,
for all $f\in \mathsf{Hom}^{(A)}(M,M')$.
\end{itemize}
A morphism $(Q,\vartheta)\to (Q',\vartheta')$ is an $R$-bimodule map $u:Q\to
Q'$ such that 
$$
\vartheta'_M(u\ox_R M)=(M\ox_R u)\vartheta_M,\qquad
\textrm{for all}\ M\in \mathsf{M}^{(A)}.
$$
\end{example}

In order to give full embeddings of $\mathsf{M}^{(A)}$ into the category
$(\mathsf{M}_{(A)})^{U_{(A)}}$ in Example \ref{ex:mod}, and of
$\mathsf{M}_{(A)}$ into the category ${}^{U^{(A)}}(\mathsf{M}^{(A)})$ in Example 
\ref{ex:comod}, some preparation is needed. 

\begin{lemma}\label{lem:phi}
Let $A$ be a regular weak multiplier bialgebra with right full
comultiplication. For any idempotent non-degenerate right $A$-module
$(V,\cdot)$ and any full right $A$-comodule $(M,\lambda,\varrho)$, there is an 
$R$-bimodule map 
$$
\varphi_{V,M}:V\ox M \to M\ox V, 
\qquad v\cdot a \ox m \mapsto (M\ox v\cdot(-))\varrho(m\ox a)
$$
obeying the following properties.
\begin{itemize}
\item[{(1)}] For the idempotent $\iota_{M,V}\pi_{M,V}$ in Lemma
\ref{lem:comod-mod_idempotent}~(1),
$\iota_{M,V}\pi_{M,V}\varphi_{V,M}=\varphi_{V,M}$. 
\item[{(2)}] It is $R$-balanced in the sense that
$\varphi_{V,M}(v\wr r \ox m)=\varphi_{V,M}(v \ox r \bl m)$, for any $v\in V$,
$m\in M$ and $r\in R$.
\item[{(3)}] For any $g\in \mathsf{Hom}_{(A)}(V,V')$, $\varphi_{V',M}(g\ox
 M)=(M\ox g)\varphi_{V,M}$.
\item[{(4)}] For any $f\in \mathsf{Hom}^{(A)}(M,M')$, $\varphi_{V,M'}(V\ox
 f)=(f \ox V)\varphi_{V,M}$.
\end{itemize}
\end{lemma}

\begin{proof}
In order to prove that $\varphi_{V,M}$ is a well-defined linear map, we need
to show that it takes a zero element to zero. So assume that $v\cdot a \ox
m=0$. Then for any $b\in A$ 
\begin{eqnarray*}
0&=&(M\ox v\cdot a(-))\lambda(m\ox b)
=(M\ox v\cdot (-))[(1\ox a)\lambda(m\ox b)]\\
&=&(M\ox v\cdot (-))[\varrho(m\ox a)(1\ox b)]
=[(M\ox v\cdot(-))\varrho(m\ox a)](1\ox b),
\end{eqnarray*}
where the third equality follows by the compatibility
condition (2.10) in \cite{Bo:wmba_comod}. Since $V$ is a non-degenerate right
$A$-module, so is $M\ox V$ by \cite[Lemma 1.11]{JaVe}. So we can simplify by
$b$ and conclude that $(M\ox v\cdot(-))\varrho(m\ox a)=0$ and hence
$\varphi_{V,M}$ is a well-defined linear map. It is a left and right
$R$-module map by \cite[Lemma 4.9~(4) and (6)]{Bo:wmba_comod}, respectively.

(1) For any $v\in V$, $m\in M$ and $a\in A$, using the implicit summation
index notation $\varrho(m\ox a)=:m^\varrho\ox a^\varrho$,
\begin{eqnarray*}
\iota_{M,V}\pi_{M,V}\varphi_{V,M}(v\cdot a \ox m)
&=&(m^\varrho\br(-) \ox v\cdot (-))((1\ox a^\varrho)E)\\
&=&m^\varrho\ox v\cdot a^\varrho
=\varphi_{V,M}(v\cdot a \ox m).
\end{eqnarray*}
The first equality follows by Lemma \ref{lem:comod-mod_idempotent}~(1) and the
second one follows by \cite[Lemma 4.10~(3)]{Bo:wmba_comod}. 

(2) follows by \cite[Lemma 4.9~(3)]{Bo:wmba_comod}. 

(3) For any $g\in \mathsf{Hom}_{(A)}(V,V')$, $v\in V$, $m\in M$ and $a\in A$,
\begin{eqnarray*}
\varphi_{V',M}(g\ox M)&&\hspace{-1cm}(v\cdot a \ox m)
=\varphi_{V',M}(g(v)\cdot a \ox m)
=(M\ox g(v)\cdot (-))\varrho(m\ox a))\\
&=&(M\ox g)((M\ox v\cdot(-))\varrho(m\ox a))
=(M\ox g)\varphi_{V,M}(v\cdot a\ox m).
\end{eqnarray*} 

(4) For any $f \in \mathsf{Hom}^{(A)}(M,M')$, $v\in V$, $m\in M$ and $a\in A$,
\begin{eqnarray*}
\varphi_{V,M'}(V\ox f)&&\hspace{-1cm}(v\cdot a\ox m)
=(M'\ox v\cdot (-))(\varrho'(f\ox A)(m\ox a))\\
&=&(f\ox v\cdot(-))\varrho(m\ox a))
=(f\ox V)\varphi_{V,M}(v\cdot a\ox m).
\end{eqnarray*} 
\end{proof}

\begin{example}\label{ex:phi_AM}
For a regular weak multiplier bialgebra $A$ with right full comultiplication,
for any full right $A$-comodule $(M,\lambda,\varrho)$ and the right $A$-module
$A$ with action provided by the multiplication,
$$
\varphi_{A,M}(ba\ox m)=
(1\ox b)\varrho(m\ox a)=
\varrho(m\ox ba),
$$
for any $m\in M$ and $a,b\in A$, where in the second equality we used that
$\varrho$ is a left $A$-module map, cf. \cite[Proposition
2.1~(1)]{Bo:wmba_comod}. Hence by the idempotency of $A$,
$\varphi_{A,M}(a\ox m)=\varrho(m\ox a)$.
\end{example}

\begin{example}\label{ex:phi_VA}
For a regular weak multiplier bialgebra $A$ with right full comultiplication,
for any idempotent non-degenerate right $A$-module $(V,\cdot)$ and the right
$A$-comodule $(A,T_1,T_3)$, 
$$
\varphi_{V,A}(v\cdot a\ox b)=
(A\ox v\cdot(-))T_3(b\ox a),
$$
for any $v\in V$ and $a,b\in A$. 
\end{example}

Under the assumptions and using the notation in Lemma \ref{lem:phi}, it
follows from Lemma \ref{lem:phi}~(2) that $\varphi_{V,M}$ projects to an
$R$-bimodule map
\begin{equation}\label{eq:projected}
\widehat \varphi_{V,M}:V\ox_R M \to M\ox_R V, 
\quad v\cdot a \ox_R m \mapsto 
\pi_{M,V}(M\ox v\cdot(-))\varrho(m\ox a)=
m^\varrho\ox_R v\cdot a^\varrho
\end{equation}
where $\pi_{M,V}:M\ox V \to M\ox_R V$ is the canonical epimorphism and we used
the notation $\varrho(m\ox a)=:m^\varrho \ox a^\varrho$ with implicit
summation understood. In other words, $\widehat \varphi_{V,M}$ is defined by
the equality $\widehat \varphi_{V,M}\pi_{V,M}=\pi_{M,V}
\varphi_{V,M}$. Applying Lemma \ref{lem:phi}~(1) in the last equality, also 
\begin{equation}\label{eq:iota-phi}
\iota_{M,V}\widehat\varphi_{V,M}=
\iota_{M,V}\widehat\varphi_{V,M}\pi_{V,M}\iota_{V,M}=
\iota_{M,V}\pi_{M,V}\varphi_{V,M}\iota_{V,M}=
\varphi_{V,M}\iota_{V,M}.
\end{equation}
 
\begin{lemma}\label{lem:phi_hat}
Let $A$ be a regular weak multiplier bialgebra with right full
comultiplication. For any idempotent non-degenerate right $A$-module
$(V,\cdot)$ and any full right $A$-comodule $(M,\lambda,\varrho)$, the map
$\widehat \varphi_{V,M}$ in \eqref{eq:projected} obeys the following
properties.
\begin{itemize}
\item[{(1)}] $\widehat \varphi_{R,M}(r\ox_R m \br s)=r\bl m\ox_R s$.
\item[{(2)}] $\widehat \varphi_{V,R}(r\wl v\ox_R s)=r\ox_R v\wr s$.
\item[{(3)}] $(\widehat \varphi_{V,M} \ox_R W)(V\ox_R \widehat \varphi_{W,M})=
\widehat \varphi_{V\ox_R W,M}$.
\item[{(4)}] $(M\ox_R \widehat \varphi_{V,N})(\widehat \varphi_{V,M}\ox_R N)=
\widehat \varphi_{V,M\ox_R N}$.
\end{itemize}
\end{lemma}

\begin{proof}
(1) Since the $A$-action $r\ox a \mapsto r\cdot a =\sqcap^R(ra)$ on $R$ is
surjective by \cite[Proposition 5.3]{BoGTLC:wmba}, the claim follows by 
\begin{eqnarray*}
\widehat \varphi_{R,M}(r\cdot a\ox_R m\br s)
&=&(m\br s)^\varrho\ox_R \sqcap^R(ra^\varrho)
=m^\varrho\ox_R \sqcap^R(ra^\varrho s)\\
&=&m^\varrho\ox_R \sqcap^R((ra)^\varrho s)
=m^\varrho\ox_R \sqcap^R((ra)^\varrho) s\\
&=&m^\varrho\br \sqcap^R((ra)^\varrho) \ox_R s
=\sqcap^R(ra)\bl m \ox_R s
=(r\cdot a)\bl m \ox_R s,
\end{eqnarray*} 
for any $m\in M$, $a\in A$ and $s,r\in R$. In the first equality we applied the
definition of $\widehat\varphi_{R,M}$ via \eqref{eq:projected}. The second
equality follows by \cite[Lemma 4.9~(6)]{Bo:wmba_comod}. In the third equality
we used that $\varrho$ is a left $R$-module map being a left $A$-module map by
\cite[Proposition 2.1~(1)]{Bo:wmba_comod}. The fourth equality is a consequence
of (3.8) in \cite{BoGTLC:wmba} and the penultimate one follows by \cite[Lemma
4.11~(1)]{Bo:wmba_comod}. 

(2) Since the $A$-action on $V$ is surjective by assumption, the claim follows
by
\begin{eqnarray*}
\widehat \varphi_{V,R}(\sqcap^R(b)\wl (v\cdot a) \ox_R s)
&=&\pi_{R,V} (R\ox v\cdot(-))[(1\ox a\overline \sqcap^L(b)s)E]\\
&=&\pi_{R,V} (\sqcap^R(b)(-)\ox v\cdot(-))[(1 \ox as)E]\\
&=&
\sqcap^R(b)\ox_R (v\cdot a)\wr s,
\end{eqnarray*} 
for any $v\in V$, $a\in A$ and $r,s\in R$. The second equality follows by
\cite[Lemma 3.5]{BoGTLC:wmba} and \cite[Lemma 3.9]{BoGTLC:wmba}. In the last
equality we used Lemma \ref{lem:comod-mod_idempotent}~(1).

(3) Since the $R$-actions on $W$ and on $V\ox_R W$ are surjective (by
assumption, and by \cite[Proposition 5.5]{BoGTLC:wmba}, respectively), the
claim follows by 
\begin{eqnarray*}
(\widehat \varphi_{V,M} \ox_R W)&&\hspace{-1cm}
(V\ox_R \widehat \varphi_{W,M})((v\ox_R w\cdot b)\cdot a\ox_R m)\\
&=&\pi_{M,V,W}(M\ox v\cdot(-)\ox w\cdot(-))(\varrho\ox A)\varrho^{13}(M\ox T_3)
(m\ox a\ox b)\\
&=&\pi_{M,V,W}(M\ox v\cdot(-)\ox w\cdot(-))(M\ox T_3)(\varrho\ox A)
(m\ox a\ox b)\\
&=&\widehat \varphi_{V\ox_R W,M}((v\ox_R w\cdot b)\cdot a \ox_R m),
\end{eqnarray*} 
for any $v\in V$, $w\in W$, $m\in M$ and $a,b\in A$. The second equality
follows by the coassociativity condition (2.14) in \cite{Bo:wmba_comod} and
the first and the last equalities follow by the form of the right $A$-action
on $V\ox_R W$ in \cite[Proposition 5.5]{BoGTLC:wmba}. 

(4) Since the $A$-action on $V$ is surjective by assumption, the claim follows
by 
\begin{eqnarray*}
(M\ox_R &&\hspace{-1cm}\widehat \varphi_{V,N})(\widehat \varphi_{V,M}\ox_R N)
(v\cdot a\ox_R m\ox_R n)\\
&=& \pi_{M,N,V}(M\ox N\ox v\cdot(-))(M\ox \varrho_N)\varrho_M^{13}
(m\ox n\ox a)\\
&=&\pi_{M,N,V}(M\ox N\ox v\cdot(-))(M\ox \varrho_N)\varrho_M^{13}
(\iota_{M,N}\pi_{M,N}\ox A)(m\ox n\ox a)\\
&=&\pi_{M,N,V}(M\ox N\ox v\cdot(-)) (\iota_{M,N}\ox A)\varrho_{M\ox_R N}
(m\ox_R n\ox a)\\
&=&\pi_{M\ox_R N,V}(M\ox_R N\ox v\cdot(-))\varrho_{M\ox_R N}(m\ox_R n\ox a)\\
&=&\widehat \varphi_{V,M\ox_R N}(v\cdot a\ox_R m\ox_R n),
\end{eqnarray*} 
for any $v\in V$, $m\in M$, $n\in N$ and $a\in A$. The second equality follows
by \cite[Lemma 5.3~(8)]{Bo:wmba_comod} and the third one follows by the
construction of $\varrho_{M\ox_R N}$ in \cite[Proposition 5.4]{Bo:wmba_comod}.
\end{proof}

\begin{remark}\label{rem:phihat_iso}
Let $A$ be a regular weak multiplier Hopf algebra in the sense of \cite{VDaWa}
(that is, let it be a regular weak multiplier bialgebra with left and right
full comultiplication such that both $A$ and $A^\op$ possess an antipode). For
any idempotent non-degenerate right $A$-module $(V,\cdot)$ and for any full
right $A$-comodule $(M,\lambda,\varrho)$, $\widehat\varphi_{V,M}$ in
\eqref{eq:projected} is an isomorphism (of vector spaces).
\end{remark}

\begin{proof}
Recall from \cite[Proposition 4.3]{VDaWa} that under the assumptions that we
made on $A$, the antipode restricts to a vector space isomorphism $S:A\to
A$. We use its inverse $S^{-1}$ to construct the to-be-inverse 
$$
(\widehat\varphi_{V,M})^{-1}:M\ox_R V \to V\ox_R M,\qquad
m\ox_R v\cdot a\mapsto v\cdot S^{-1}(S(a)^\lambda) \ox_R m^\lambda,
$$
where the implicit summation index notation $\lambda(m\ox a)=m^\lambda\ox
a^\lambda$ is used. Let us see first that $(\widehat\varphi_{V,M})^{-1}$ is a
well-defined linear map. Any element of $M\ox_R V$ is a linear combination of
elements of the form $m\ox_R v\cdot a \overline \sqcap^L(bc)=m\ox_R
\sqcap^R(bc)\wl (v\cdot a)$, for $m\in M$, $v\in V$ and $a,b,c\in A$. Assume
that for some elements $m^i\in M$, $v^i\in V$ and $a^i,b^i,c^i\in A$, the
finite sum $\sum_i m^i\ox_R v^i\cdot a^i \overline \sqcap^L(b^ic^i)$ is equal to
zero. Then applying the section $\iota_{M,V}$ of the canonical epimorphism
$M\ox V \to M\ox_R V$, using (4.5) in \cite{BoGTLC:wmba} and the implicit
summation index notation $T_3(c\ox b)=:c^3\ox b^3$, and omitting the summation
symbol for brevity, we obtain 
$$
0=m^i\br \sqcap^R(c^{i3})\ox \sqcap^R(b^{i3})\wl(v^i\cdot a^i)
=m^i\br \sqcap^R(c^{i3})\ox v^i\cdot a^i\overline \sqcap^L(b^{i3}).
$$
Hence using also the index notation $\varrho(v\ox a)=v^\varrho \ox a^\varrho$,
where implicit summation is understood, it follows for any $d\in A$ that
\begin{eqnarray*}
0&=&v^i\cdot a^i\overline \sqcap^L(b^{i3}) S^{-1}(S(d)^\varrho)\ox 
(m^i\br \sqcap^R(c^{i3}))^\varrho\\
&=&v^i\cdot a^i\overline \sqcap^L(b^{i3}) S^{-1}(S(d)^\varrho
\sqcap^R(c^{i3}))\ox m^{i\varrho}\\
&=&v^i\cdot a^i\overline \sqcap^L(b^{i3}) \overline \sqcap^L(c^{i3})
S^{-1}(S(d)^\varrho )\ox m^{i\varrho}\\
&=&v^i\cdot a^i\overline \sqcap^L(b^{i3} \overline \sqcap^L(c^{i3}))
S^{-1}(S(d)^\varrho )\ox m^{i\varrho}\\
&=&v^i\cdot a^i\overline \sqcap^L(b^ic^i)S^{-1}(S(d)^\varrho )\ox m^{i\varrho}\\
&=&v^i\cdot S^{-1}(S(d)^\varrho S(a^i\overline \sqcap^L(b^ic^i)))\ox m^{i\varrho}\\
&=&v^i\cdot S^{-1}(S(d)S(a^i\overline \sqcap^L(b^ic^i))^\lambda)\ox m^{i\lambda}\\
&=&v^i\cdot S^{-1}(S(a^i\overline \sqcap^L(b^ic^i))^\lambda)d\ox m^{i\lambda}.
\end{eqnarray*}
The second equality holds by \cite[Lemma 4.9~(6)]{Bo:wmba_comod}, the third
one holds by \cite[Lemma 6.14]{BoGTLC:wmba}, the fourth one does by \cite[Lemma
3.4]{BoGTLC:wmba} and the fifth one holds by \cite[Lemma 3.7~(1)]{BoGTLC:wmba}. 
The sixth and the last equalities follow by the anti-multiplicativity of $S$,
cf. \cite[Theorem 6.12]{BoGTLC:wmba}. The penultimate equality follows by the
compatibility condition (2.10) in \cite{Bo:wmba_comod}. Simplifying by $d$,
this proves that $v^i\cdot S^{-1}(S(a^i\overline \sqcap^L(b^ic^i))^\lambda)\ox
m^{i\lambda}$ is equal to zero. Thus applying the canonical epimorphism $V\ox M
\to V\ox_R M$, 
$$
0=
v^i\cdot S^{-1}(S(a^i\overline \sqcap^L(b^ic^i))^\lambda)\ox_R m^{i\lambda}=
(\widehat\varphi_{V,M})^{-1}(m^i\ox_R v^i\cdot a^i\overline \sqcap^L(b^ic^i))
$$
as needed. We turn to proving that $(\widehat\varphi_{V,M})^{-1}$ is indeed
the inverse of $\widehat\varphi_{V,M}$. By (2.3) in \cite{BoGTLC:wmba},
$E(1\ox a)\in R\ox A$ for any $a\in A$. Hence there is a well-defined linear
map
$$
E_1^{M,A}:M\ox A \to M\ox A,\qquad 
m\ox a \mapsto ((-)\bl m \ox A)[E(1\ox a)].
$$
By \cite[Lemma 6.4~(2)]{Bo:wmba_comod}, for any $m\in M$ and $a,b\in A$ the
identity $m^{\varrho\lambda}\ox S(a^\varrho)^\lambda b=
E_1^{M,A}(m\ox S(a)b)$ holds. So simplifying by $b$, 
\begin{equation}\label{eq:phihat_iso_1}
\lambda (M\ox S)\varrho=E_1^{M,A}(M\ox S).
\end{equation}
Using \eqref{eq:phihat_iso_1} in the second equality,
\begin{eqnarray*} 
(\widehat\varphi_{V,M})^{-1}\widehat\varphi_{V,M}(v\cdot a \ox_R m)&=&
v\cdot S^{-1}(S(a^\varrho)^\lambda)\ox_R m^{\varrho\lambda}\\
&=&\pi_{V,M}(v\cdot (-)\ox (-) \bl m)(S^{-1}\ox R)[E^{21}(S(a)\ox 1)]\\
&=&\pi_{V,M}(v\cdot (-)\ox (-) \bl m)[(a\ox 1)F]=
v\cdot a\ox_R m.
\end{eqnarray*}
The third equality holds since by \cite[Lemma 6.14]{BoGTLC:wmba} the
restriction of the anti-multiplicative map $\overline S\,^{-1}$ to $L$ is
equal to $\overline \sigma$, hence it follows by \eqref{eq:F^op} that
$(R\ox S^{-1})[E(1\ox S(a))]=(1\ox a)F^{21}$. 
The last equality follows by Lemma
\ref{lem:comod-mod_idempotent}~(2). Symmetrically,
\begin{eqnarray*} 
\widehat\varphi_{V,M}(\widehat\varphi_{V,M})^{-1}(m\ox_R v\cdot a)&=&
m^{\lambda\varrho}\ox_R v\cdot S^{-1}(S(a)^\lambda)^\varrho\\
&=&\pi_{M,V}(m\br(-)\ox v\cdot (-))[(1\ox a)E]=
m\ox_R v\cdot a,
\end{eqnarray*}
where the second equality follows by applying \eqref{eq:phihat_iso_1} to the
opposite weak multiplier Hopf algebra $A^{\mathsf{op}}$ and its right comodule
$(M,\varrho,\lambda)$; and the last equality follows by Lemma
\ref{lem:comod-mod_idempotent}~(1).
\end{proof}

\begin{proposition}\label{prop:I^{(A)}}
For any regular weak multiplier bialgebra $A$ with right full
comultiplication, the functor $U^{(A)}:\mathsf{M}^{(A)} \to {}_R \mathsf{M}_R$
in \cite[Theorem 5.7]{Bo:wmba_comod} (see Section \ref{sec:comodule})
factorizes through an appropriate strict monoidal functor
$I^{(A)}:\mathsf{M}^{(A)} \to (\mathsf{M}_{(A)})^{U_{(A)}}$ (via the evident
forgetful functor $(\mathsf{M}_{(A)})^{U_{(A)}}\to {}_R \mathsf{M}_R$). 
\end{proposition}

\begin{proof}
The object map of the stated functor $I^{(A)}$ takes a full right $A$-comodule
$(M,\lambda,\varrho)$ to the firm $R$-bimodule $M$ (with the actions in
\eqref{eq:b_actions}) and the family of $R$-bimodule maps $\widehat 
\varphi_{V,M}:V\ox_R M \to M\ox_R V$ in \eqref{eq:projected}, for any
idempotent non-degenerate right $A$-module $(V,\cdot)$. Let us see that it
obeys the conditions in Example \ref{ex:mod}. The first compatibility
condition with the monoidal unit $R$ holds by Lemma \ref{lem:phi_hat}~(1). 
The the second compatibility condition with the monoidal product $\ox_R$
follows by Lemma \ref{lem:phi_hat}~(3). Naturality of $\widehat\varphi_{V,M}$
in $V$ follows by Lemma \ref{lem:phi}~(3), proving that 
$(M,\widehat\varphi_{-,M})$ is an object of $(\mathsf{M}_{(A)})^{U_{(A)}}$. The 
stated functor $I^{(A)}$ acts on the morphisms as the identity map. Indeed,
any morphism $f:M\to M'$ in $\mathsf{M}^{(A)}$ is a morphism
$(M,\widehat\varphi_{-,M})\to (M',\widehat\varphi_{-,M'})$ in
$(\mathsf{M}_{(A)})^{U_{(A)}}$ by Lemma \ref{lem:phi}~(4). The functor
$I^{(A)}$ is strict monoidal by Lemma \ref{lem:phi_hat}~(2) and (4). 
\end{proof}

\begin{proposition}\label{prop:I^{(A)}_ff}
For any regular weak multiplier bialgebra $A$ with right full
comultiplication, the functor $I^{(A)}:\mathsf{M}^{(A)} \to
(\mathsf{M}_{(A)})^{U_{(A)}}$ in Proposition \ref{prop:I^{(A)}} is fully 
faithful. 
\end{proposition}

\begin{proof}
Since $I^{(A)}$ acts on the morphisms as the identity map, it is evidently
faithful. In order to see that is full as well, take a morphism
$f:I^{(A)}(M,\lambda,\varrho) \to I^{(A)}(M',\lambda',\varrho')$ in
$(\mathsf{M}_{(A)})^{U_{(A)}}$. This means that $f$ is an $R$-bimodule map
satisfying 
$$
(f\ox_R V)\widehat \varphi_{V,M}=\widehat \varphi_{V,M'} (V\ox_R f),
$$
for any object $(V,\cdot)$ in $\mathsf{M}_{(A)}$. Composing both sides of this
equality by $\pi_{V,M}$ on the right and by $\iota_{M',V}$ on the left and
applying Lemma \ref{lem:phi}~(1), we obtain
$$
(f\ox V) \varphi_{V,M}=
\varphi_{V,M'} (V\ox f),
$$
for any object $(V,\cdot)$ in $\mathsf{M}_{(A)}$; so in particular for the right
$A$-module $(A,\mu)$. Thus applying Example \ref{ex:phi_AM}, we conclude that
$$
(f\ox A) \varrho=
\varrho' (f\ox A),
$$
that is, that $f$ is a morphism of $A$-comodules. 
\end{proof}

Summarizing Proposition \ref{prop:I^{(A)}} and Proposition
\ref{prop:I^{(A)}_ff}, we proved the following. 

\begin{theorem}\label{thm:comod_embed}
For any regular weak multiplier bialgebra $A$ with right full
comultiplication, the functor $U^{(A)}:\mathsf{M}^{(A)} \to {}_R \mathsf{M}_R$
in \cite[Theorem 5.7]{Bo:wmba_comod} (see Section \ref{sec:comodule})
factorizes through the strict monoidal fully faithful functor
$I^{(A)}:\mathsf{M}^{(A)} \to (\mathsf{M}_{(A)})^{U_{(A)}}$ in Proposition
\ref{prop:I^{(A)}} (via the evident forgetful functor
$(\mathsf{M}_{(A)})^{U_{(A)}}\to {}_R \mathsf{M}_R$).
\end{theorem}

One can proceed symmetrically if interchanging the roles of modules and
comodules. 

\begin{proposition}\label{prop:I_{(A)}}
For any regular weak multiplier bialgebra $A$ with right full
comultiplication, the functor $U_{(A)}:\mathsf{M}_{(A)} \to {}_R \mathsf{M}_R$
in \cite[Theorem 5.6]{BoGTLC:wmba} (see Section \ref{sec:module}) factorizes
through an appropriate strict monoidal functor $I_{(A)}:\mathsf{M}_{(A)} \to
{}^{U^{(A)}}(\mathsf{M}^{(A)})$ (via the evident forgetful functor
${}^{U^{(A)}}(\mathsf{M}^{(A)})\to {}_R \mathsf{M}_R$).
\end{proposition}

\begin{proof}
The object map of the stated functor $I_{(A)}$ takes an idempotent
non-degenerate right $A$-module $(V,\cdot)$ to the firm $R$-bimodule $V$ (with
the actions in \eqref{eq:w_actions}) and the family of $R$-bimodule maps
$\widehat \varphi_{V,M}:V\ox_R M \to M\ox_R V$ in \eqref{eq:projected}, for any
full right $A$-comodule $(M,\lambda,\varrho)$. Let us see that it
obeys the conditions in Example \ref{ex:comod}. The first compatibility
condition with the monoidal unit $R$ holds by Lemma \ref{lem:phi_hat}~(2). 
The the second compatibility condition with the monoidal product $\ox_R$
follows by Lemma \ref{lem:phi_hat}~(4). Naturality of $\widehat\varphi_{V,M}$
in $M$ follows by Lemma \ref{lem:phi}~(4), proving that
$(V,\widehat\varphi_{V,-})$ is an object of ${}^{U^{(A)}}(\mathsf{M}^{(A)})$. The 
stated functor $I_{(A)}$ acts on the morphisms as the identity map. Indeed,
any morphism $g:V\to V'$ in $\mathsf{M}_{(A)}$ is a morphism
$(V,\widehat\varphi_{V,-})\to (V',\widehat\varphi_{V',-})$ in
${}^{U^{(A)}}(\mathsf{M}^{(A)})$ by Lemma \ref{lem:phi}~(3). The functor
$I_{(A)}$ is strict monoidal by Lemma \ref{lem:phi_hat}~(1) and (3). 
\end{proof}

\begin{proposition}\label{prop:I_{(A)}_ff}
For any regular weak multiplier bialgebra $A$ with right full
comultiplication, the functor $I_{(A)}:\mathsf{M}_{(A)} \to {}^{U^{(A)}}
(\mathsf{M}^{(A)})$ in Proposition \ref{prop:I_{(A)}} is fully faithful.
\end{proposition}

\begin{proof}
Since $I_{(A)}$ acts on the morphisms as the identity map, it is evidently
faithful. In order to see that it is full as well, take a morphism
$g:I_{(A)}(V,\cdot) \to I_{(A)}(V',\cdot)$ in ${}^{U^{(A)}} (\mathsf{M}^{(A)})$. This
means an $R$-bimodule map $g$ satisfying
$$
(M\ox_R g)\widehat \varphi_{V,M}=\widehat\varphi_{V,M}(g\ox_R M)
$$
for any object $(M,\lambda,\varrho)$ in $\mathsf{M}^{(A)}$. Composing both
sides of this equality by $\pi_{V,M}$ on the right and by $\iota_{M,V'}$ on
the left and applying Lemma \ref{lem:phi}~(1), we obtain 
$$
(M\ox g) 
\varphi_{V,M}=
\varphi_{V',M}(g\ox M)
$$
for any object $(M,\lambda,\varrho)$ in $\mathsf{M}^{(A)}$; so in particular for
$(M,\lambda,\varrho)=(A,T_1,T_3)$. Composing both sides of the resulting
equality with $\epsilon \ox V'$ and using that by the counitality axiom 
of weak multiplier bialgebra, as appearing in (1.3) in \cite{Bo:wmba_comod},
the map $\varphi_{V,A}$ in Example \ref{ex:phi_VA} obeys 
\begin{equation}\label{eq:epsilon_phi}
(\epsilon \ox V)\varphi_{V,A}(v\ox a)=v\cdot a, \qquad \forall v\in A,\ a\in A,
\end{equation}
we conclude that $g(v\cdot a)=g(v)\cdot a$ for all $v\in V$ and $a\in A$. That
is, $g$ is a morphism of $A$-modules. 
\end{proof}

Summarizing Proposition \ref{prop:I_{(A)}} and Proposition
\ref{prop:I_{(A)}_ff}, we proved the following. 

\begin{theorem}\label{thm:mod_embed}
For any regular weak multiplier bialgebra $A$ with right full
comultiplication, the functor $U_{(A)}:\mathsf{M}_{(A)} \to {}_R \mathsf{M}_R$
in \cite[Theorem 5.6]{BoGTLC:wmba} (see Section \ref{sec:module}) 
factorizes through the strict monoidal fully faithful functor
$I_{(A)}:\mathsf{M}_{(A)} \to {}^{U^{(A)}}(\mathsf{M}^{(A)})$ in Proposition
\ref{prop:I_{(A)}} (via the evident forgetful functor
${}^{U^{(A)}}(\mathsf{M}^{(A)})\to {}_R \mathsf{M}_R$). 
\end{theorem}

In contrast to usual, unital weak bialgebras in \cite{WHAI,Nill}, the functors
in Theorem \ref{thm:comod_embed} and Theorem \ref{thm:mod_embed} do not seem
to be equivalences.

\section{Yetter-Drinfeld modules}\label{sec:YD}

The aim of this section is to find the proper notion of Yetter-Drinfeld module
over a regular weak multiplier bialgebra with right full comultiplication. 

For a usual, unital (weak) bialgebra $A$, the isomorphism $I_A$ between the
category of $A$-modules; and the left commutant of the forgetful functor from
the category of $A$-comodules, induces an isomorphism $J_A$ between the
category of Yetter-Drinfeld $A$-modules and the left weak center of the
category of $A$-comodules. Symmetrically, the isomorphism $I^A$, between the
category of $A$-comodules; and the right commutant of the forgetful functor
from the category of $A$-modules, induces an isomorphism $J^A$ between the
category of Yetter-Drinfeld $A$-modules and the right weak center of the
category of $A$-modules. However, as we have seen in the previous section, for
a weak multiplier bialgebra $A$, the analogous functors $I_{(A)}$ and
$I^{(A)}$ are no longer isomorphisms but fully faithful embeddings. In this
section we show that they induce fully faithful embeddings $J_{(A)}$ and
$J^{(A)}$ of the category of appropriately defined Yetter-Drinfeld $A$-modules
into the left weak center of the category of full $A$-comodules, and into the
right weak center of the category of idempotent non-degenerate $A$-modules,
respectively. 

Let $A$ be a weak multiplier bialgebra and let $V$ and $W$ be idempotent
non-degenerate right $A$-modules. Consider the map 
\begin{equation}\label{eq:E_2^VW}
E_2^{V,W}:V\ox W \to V\ox W,\qquad
v\cdot a\ox w\cdot b\mapsto (v\cdot(-) \ox w\cdot(-))[(a\ox b)E]
\end{equation}
in \cite[Lemma 5.4]{BoGTLC:wmba}. Regard $V\ox W$ as a right $A$-module via
the (so-called {\em diagonal}) action
\begin{equation}\label{eq:diag_action}
(v\cdot a\ox w\cdot b)\cdot c=
(v\cdot(-)\ox w\cdot(-))[(a\ox b)\Delta(c)], 
\end{equation}
cf. \cite[Proposition 5.5]{BoGTLC:wmba}. Then for any $v\in V$, $w\in W$ and
$c\in A$,
$$
E_2^{V,W}((v\ox w)\cdot c)=(v\ox w)\cdot c= E_2^{V,W}(v\ox w)\cdot c
$$ 
so in particular $E_2^{V,W}$ is a right $A$-module map.

Throughout the section we use the notation
$$
(v\ox w)\Delta(a):=
(v\ox w)\cdot a
\qquad \textrm{and} \qquad
(v\ox w)\Delta^{\mathsf{op}}(a):=
\mathsf{tw}[(w\ox v)\cdot a].
$$

\begin{lemma}\label{lem:YD_actions}
Let $A$ be a regular weak multiplier bialgebra with right full
comultiplication. Let $X$ be a vector space which carries the structure of an
idempotent non-degenerate right $A$-module $(X,\cdot)$ and the structure of a
full right $A$-comodule $(X,\lambda,\varrho)$. Denote the corresponding
$R$-actions on $X$ in \eqref{eq:w_actions} by $\wl$ and $\wr$, and those in
\eqref{eq:b_actions} by $\bl$ and $\br$.
\begin{itemize}
\item[{(1)}] The following assertions are equivalent.
\begin{itemize}
\item[{(1.a)}] The right $R$-actions $\wr$ and $\br$ are equal. That is, for
 any $x\in X$ and $a,b\in A$, $x\cdot a \overline \sqcap^R(b)=(X\ox
 \epsilon)\lambda(x\cdot a \ox b)$.
\item[{(1.b)}] $E_2^{X,A}\varrho=\varrho$ (where $E_2^{X,A}$ is as in
 \eqref{eq:E_2^VW}).
\end{itemize}
\item[{(2)}] The following assertions are equivalent.
\begin{itemize}
\item[{(2.a)}] The left $R$-actions $\wl$ and $\bl$ are equal. That is, for
 any $x\in X$ and $a,b\in A$, $x\cdot a \overline \sqcap^L(b)=(X\ox
 \epsilon)\varrho(x\cdot a \ox b)$.
\item[{(2.b)}] $\varrho (E_2^{A,X})^{21}=\varrho$ (where $E_2^{A,X}$ is as in
 \eqref{eq:E_2^VW}).
\end{itemize}
\item[{(3)}] If $\varrho [(x\ox a)\Delta^{\mathsf{op}}(b)]= 
\varrho(x\ox a)\Delta(b)$ for all $x\in X$ and $a,b\in A$, then the assertions
in part (1) and part (2) are equivalent also to each other. 
\end{itemize}
\end{lemma}

\begin{proof}
(1) From \eqref{eq:E_2^VW} it follows that $E_2^{X,A}(x\ox a)=(x\wr(-)\ox
A)[(1\ox a)E]$. Hence assertion (1.a) implies (1.b) by \cite[Lemma
4.10~(3)]{Bo:wmba_comod}. 

By (1.b) and \cite[Lemma 3.9]{BoGTLC:wmba}, for any $y\in X$ and $a,b\in A$, 
$$
((-)\wr \overline \sqcap^R(a)\ox A)\varrho(y\ox b)=
\varrho(y\ox b)(1\ox \sqcap^L(a)). 
$$
Using this identity in the first equality and \cite[Lemma
4.9~(5)]{Bo:wmba_comod} in the second one, it follows for any $\omega\in
\mathsf{Lin}(A,k)$ that
$$
((X\ox \omega)\varrho(y\ox b))\wr\overline \sqcap^R(a)=
(X\ox \omega)[\varrho(y\ox b) (1\ox \sqcap^L(a))]=
((X\ox \omega)\varrho(y\ox b))\br\overline \sqcap^R(a).
$$
Since $X$ is a full right $A$-comodule, any element of $X$ can be written as a 
linear combination of elements of the form $(X\ox \omega)\varrho(y\ox b)$,
proving that (1.a) holds.

(2) For any $a,b,c\in A$, it follows by \cite[Proposition
2.5~(1)]{BoGTLC:wmba} that
$$
E_2^{21}(a\ox bc)=
(a\ox 1)((\overline \sqcap^L \ox A)T_2^{21}(c\ox b)).
$$
Using this identity in the second equality and introducing the implicit
summation index notation $T_2(b\ox c)=:b^2\ox c^2$, we see that (2.a) implies 
\begin{eqnarray*}
\varrho (E_2^{A,X})^{21}&&\hspace{-1.3cm}(x\cdot a \ox bc)=
\varrho (x\cdot (-)\ox A)E_2^{21}(a\ox bc)=
\varrho (x\cdot a\overline \sqcap^L (c^2) \ox b^2)\\
&\stackrel{\mathrm{(2.a)}}=&\varrho (\sqcap^R (c^2) \bl (x\cdot a) \ox b^2)=
\varrho (x\cdot a \ox b^2\sqcap^R (c^2))=
\varrho (x\cdot a \ox bc);
\end{eqnarray*}
that is, (2.b). The penultimate equality follows by \cite[Lemma
4.9~(3)]{Bo:wmba_comod} and the last one follows by \cite[Lemma
3.7~(4)]{BoGTLC:wmba}.

If (2.b) holds then for any $x\in X$ and $a,b,c\in A$
\begin{eqnarray*}
\varrho(\sqcap^R(c)\wl (x\cdot a)\ox b)=
\varrho(x\cdot a\overline \sqcap^L (c)\ox b)=
\varrho(x\cdot a\ox b\sqcap^R (c))=
\varrho(\sqcap^R(c)\bl (x\cdot a)\ox b).
\end{eqnarray*}
In the second equality we used (2.b) together with \cite[Lemma
3.9]{BoGTLC:wmba}. The last equality follows by \cite[Lemma
4.9~(3)]{Bo:wmba_comod}. By surjectivity of the $A$-action on $X$, this proves
$\varrho(\sqcap^R(c)\wl x \ox b)=\varrho(\sqcap^R(c)\bl x \ox b)$. Applying
$X\ox \epsilon$ to both sides of this equality, we obtain
$$
\sqcap^R (b)\bl (\sqcap^R(c)\wl x)=\sqcap^R (b)\bl (\sqcap^R(c)\bl x),\quad
\forall x\in X, \ b,c\in A.
$$
Since $(X,\bl)$ is a firm left $R$-module by \cite[Theorem 4.5]{Bo:wmba_comod}
and $R$ has local units by \cite[Theorem 4.6~(2)]{BoGTLC:wmba}, $(X,\bl)$ is a
non-degenerate left $R$-module. So we conclude that $\sqcap^R(c)\wl x=
\sqcap^R(c)\bl x$ for all $x\in X$ and $c\in A$; that is, (2.a) holds. 

(3) Under the hypothesis in (3), for any $x\in X$ and $a,b,c\in A$
\begin{eqnarray}\label{eq:YD_actions}
(\varrho (E_2^{A,X})^{21}(x\cdot a \ox b))\Delta(c)&=&
\varrho(((x\cdot (-) \ox A)[(a \ox b)E^{21}])\Delta^{\mathsf{op}}(c))\\
&=&\varrho(x\cdot (-) \ox A)[(a \ox b)E^{21}\Delta^{\mathsf{op}}(c)]\nonumber\\
&=&\varrho(x\cdot (-)\ox A)[(a \ox b)\Delta^{\mathsf{op}}(c)]\nonumber\\
&=&\varrho((x\cdot a \ox b))\Delta(c). \nonumber
\end{eqnarray}
In the second equality we used that for any $x\in X$, 
$A\ox x\cdot(-):A\ox A \to A\ox X$ is a morphism of right $A$-modules with
respect to the diagonal actions in \eqref{eq:diag_action}.

If (1.b) holds then the range of $\varrho$ is contained in the range of
$E_2^{X,A}$, which is a non-degenerate right $A$-module via the diagonal
action \eqref{eq:diag_action} by \cite[Lemma 5.4 and Proposition
5.5]{BoGTLC:wmba}. Thus we conclude from \eqref{eq:YD_actions} that (2.b)
holds. 

By axiom (vii) of weak multiplier bialgebra in \cite[Definition
1.1]{Bo:wmba_comod}, for any given elements $a,b\in A$ there exist finitely
many elements $p^i,q^i, r^i$ such that $(b\ox a)E=\sum_i(p^i\ox
q^i)\Delta(r^i)$. Then (omitting the summation symbol for brevity), it follows
from (2.b) that for any $x\in X$
\begin{eqnarray*}
\varrho(x\cdot a \ox b)&\stackrel{\mathrm{(2.b)}}=&
\varrho (E_2^{A,X})^{21}(x\cdot a \ox b)=
\varrho((x\cdot q^i\ox p^i)\Delta^{\mathsf{op}}(r^i))=
\varrho(x\cdot q^i\ox p^i)
\Delta(r^i)
\end{eqnarray*}
so that (1.b) holds by $E_2\Delta=\Delta$. In the third equality we used our
hypothesis in (3).
\end{proof}

\begin{theorem}\label{thm:YD_def}
Consider a regular weak multiplier bialgebra $A$ with right full
comultiplication. Let $X$ be a vector space which carries the structure of an
idempotent non-degenerate right $A$-module $(X,\cdot)$ and the structure of a
full right $A$-comodule $(X,\lambda,\varrho)$. If
$E_2^{X,A}\varrho=\varrho=\varrho (E_2^{A,X})^{21}$ holds for the maps 
$E_2^{X,A}$ and $E_2^{A,X}$ as in \eqref{eq:E_2^VW}, then the following
assertions are equivalent. 
\begin{itemize}
\item[{(a)}] For any full right $A$-comodule $M$, $\widehat
\varphi_{X,M}:X\ox_R M\to M\ox_R X$ in \eqref{eq:projected} is a morphism of
$A$-comodules. 
\item[{(b)}] The datum $((X,\lambda,\varrho),\widehat\varphi_{X,-})$ is an
 object in the left weak center of $\mathsf{M}^{(A)}$. 
\item[{(c)}] For any idempotent non-degenerate right $A$-module $V$,
 $\widehat\varphi_{V,X}:V\ox_R X\to X\ox_R V$ is a morphism of $A$-modules. 
\item[{(d)}] The datum $((X,\cdot),\widehat\varphi_{-,X})$ is an object in 
the right weak center of $\mathsf{M}_{(A)}$.
\item[{(e)}] For any $x\in X$ and $a,b\in A$, 
$\varrho[(x\ox a)\Delta^{\mathsf{op}}(b)]=\varrho(x\ox a)\Delta(b)$.
\end{itemize}
\end{theorem}

\begin{proof}
The equivalences (a)$\Leftrightarrow$(b) and (c)$\Leftrightarrow$(d) are
obvious. 

(c)$\Leftrightarrow$(e) 
Regard $X$ as a right $R$-module via the action $\br=\wr$; cf. Lemma
\ref{lem:YD_actions}. For any object $V$ of $\mathsf{M}_{(A)}$, it follows by
the form of the $A$-action on $X\ox_R V$ in \cite[Proposition 
5.5]{BoGTLC:wmba} that the canonical epimorphism $\pi_{X,V}:X\ox V \to X\ox_R
V$ is a morphism of $A$-modules. By \cite[Lemma 5.4]{BoGTLC:wmba}, the
idempotent map $\iota_{X,V}\pi_{X,V}$ is equal to the $A$-module map
$E_2^{X,V}$ proving that $\iota_{X,V}:X\ox_R V \to X\ox V$ is a morphism of
$A$-modules too. Using this together with Lemma \ref{lem:phi}~(1), we see that
assertion (c) is equivalent to $\varphi_{V,X}$ being a right $A$-module
map. That is, to
\begin{equation}\label{eq:c'}
(X\ox v\cdot (-))[\varrho(x\ox a)\Delta(b)]=
(X\ox v\cdot(-))\varrho((x\ox a)\Delta^{\mathsf{op}}(b)),\ \ 
\forall v\in V,\ x\in X,\ a,b\in A.
\end{equation}
This proves (e)$\Rightarrow$(c). Conversely, applying \eqref{eq:c'} to the
$A$-module $(V,\cdot)=(A,\mu)$ and using the non-degeneracy of $A$, we
conclude that (c)$\Rightarrow$(e).

(a)$\Rightarrow$(e) 
Using \eqref{eq:iota-phi} together with the construction of the $A$-comodules
$X\ox_R M$ and $M\ox_R X$ in \cite[Proposition 5.4]{Bo:wmba_comod} for any
full right $A$-comodule $(M,\lambda_M,\varrho_M)$, (a) is
equivalent to 
$$
(M\ox \varrho)\varrho_M^{13}(\iota_{M,X}\pi_{M,X}\ox A)(\varphi_{X,M}\ox A)=
(\varphi_{X,M}\ox A)(X\ox \varrho_M)\varrho^{13}(\iota_{X,M}\pi_{X,M}\ox A).
$$
Thus by \cite[Proposition 5.1 and Lemma 5.3~(8)]{Bo:wmba_comod}, also to
\begin{equation}\label{eq:a'}
(M\ox \varrho)\varrho_M^{13}(\varphi_{X,M}\ox A)=
(\varphi_{X,M}\ox A)(X\ox \varrho_M)\varrho^{13}.
\end{equation}
Substituting $(M,\lambda_M,\varrho_M)=(A,T_1,T_3)$ in \eqref{eq:a'},
evaluating both sides on $x\cdot a\ox b\ox c$ (for $x\in X$ and $a,b,c\in A$),
and applying $\epsilon \ox X\ox A$ to the results, we obtain by the form
of the counitality axiom in a regular weak multiplier bialgebra in (1.3) in
\cite{Bo:wmba_comod} and by \eqref{eq:epsilon_phi} that
$$
\varrho((x\cdot a \ox c)\Delta^{\mathsf{op}}(b))=
\varrho(x\cdot a \ox c)\Delta(b),
$$
proving (a)$\Rightarrow$(e).

(e)$\Rightarrow$(a) 
By axiom (vii) of weak multiplier bialgebra in \cite[Definition
1.1]{Bo:wmba_comod}, for any given elements $a,b\in A$ there exist finitely
many elements $p^i,q^i,r^i\in A$ such that $(b\ox a)E=\sum_i (p^i\ox
q^i)\Delta(r^i)$. In terms of these elements, for any full right $A$-comodule
$(M,\lambda_M,\varrho_M)$, $m\in M$ and $x\in X$, 
\begin{eqnarray}\label{eq:E>a_1}
(\varphi_{X,M}\ox A)&&\hspace{-1cm}(X\ox \varrho_M)\varrho^{13}
[x\cdot a\ox m\ox b]
\\&=&
(\varphi_{X,M}\ox A)(X\ox \varrho_M)\varrho^{13}(x\cdot (-)\ox M\ox A)
[(a\ox m\ox b)E^{31}]
\nonumber\\&=&
(\varphi_{X,M}\ox A)(X\ox \varrho_M)\varrho^{13}[(x\cdot q^i\ox m\ox p^i)
\Delta(r^i)^{31}]
\nonumber\\&\stackrel{\mathrm{(e)}}=&
(\varphi_{X,M}\ox A)(X\ox \varrho_M)[\varrho^{13}(x\cdot q^i\ox m\ox p^i)
\Delta(r^i)^{13}]
\nonumber\\&=&
(M\ox (x\cdot q^i)^{\varrho'}\cdot(-)\ox A)(\varrho_M\ox A)\varrho_M^{13}
(M\ox T_3)(m\ox r^i\ox p^{i\varrho'})
\nonumber\\&=&
(M\ox (x\cdot q^i)^{\varrho'}\cdot(-)\ox A)(M\ox T_3)(\varrho_M\ox A)
(m\ox r^i\ox p^{i\varrho'})
\nonumber\\&=&
m^\varrho\ox \varrho(x\cdot q^i\ox p^i)\Delta(r^{i\varrho}).\nonumber
\end{eqnarray}
where we omitted the summation symbol for brevity; and used the implicit
summation index notation $\varrho_M(m\ox r)=:m^\varrho\ox r^\varrho$ and
$\varrho(x\ox p)=:x^{\varrho'}\ox p^{\varrho'}$. The first equality follows
by the hypothesis $\varrho=\varrho (E_2^{A,X})^{21}$ and the
penultimate equality follows by the coassociativity condition \cite[Proposition
2.1~(4.d)]{Bo:wmba_comod}. On the other hand, using the same notation, 
\begin{eqnarray}\label{eq:E>a_2}
(M\ox\hspace{-1cm}&& \varrho)\varrho_M^{13}(\varphi_{X,M}\ox A)
[x\cdot a\ox m\ox b]
\\&=&
(M\ox \varrho)(M\ox x\cdot(-)\ox A)\varrho_M^{13}(\varrho_M\ox A)(m\ox a\ox b)
\nonumber\\&=&
(M\ox \varrho)(M\ox x\cdot(-)\ox A)\varrho_M^{13}(\varrho_M\ox A)
[m\ox (a\ox b)E^{21}]
\nonumber\\&=&
(M\ox \varrho)(M\ox x\cdot(-)\ox p^i(-))(M\ox \mathsf{tw})
(\varrho_M\ox A)\varrho_M^{13}(M\ox T_3)(m\ox r^i \ox q^i)
\nonumber\\&=&
(M\ox \varrho)(M\ox x\cdot(-)\ox p^i(-))(M\ox \mathsf{tw}T_3)(\varrho_M\ox A)
(m\ox r^i \ox q^i)
\nonumber\\&=&
m^\varrho\ox \varrho((x\cdot q^i\ox p^i)\Delta^{\mathsf{op}}(r^{i\varrho})).
\nonumber
\end{eqnarray}
The second and the fourth equalities follow by part (2.b) and
part (4.d) of \cite[Proposition 2.1]{Bo:wmba_comod}, respectively. 
The expressions in \eqref{eq:E>a_1} and \eqref{eq:E>a_2} are equal by assertion
(e). This proves that \eqref{eq:a'} and hence assertion (a) holds.
\end{proof}

If the equivalent assertions in Theorem \ref{thm:YD_def} hold, then we term
$X$ a (right-right) Yetter-Drinfeld $A$-module:

\begin{definition} \label{def:YD}
Consider a regular weak multiplier bialgebra $A$ with right full
comultiplication. A {\em right-right Yetter-Drinfeld module} over $A$ is a
vector space $X$ which carries the structure of an idempotent non-degenerate
right $A$-module $(X,\cdot)$ and the structure of a full right $A$-comodule
$(X,\lambda,\varrho)$ such that for all $x\in X$ and $a,b\in A$ 
$$
\varrho[(x\ox a)\Delta^{\mathsf{op}}(b)]=\varrho(x\ox a)\Delta(b)
\quad \textrm{and}\quad
E_2^{X,A}\varrho=\varrho=\varrho (E_2^{A,X})^{21},
$$
in terms of the maps $E_2^{X,A}$ and $E_2^{A,X}$ as in \eqref{eq:E_2^VW}.

A morphism of right-right Yetter-Drinfeld modules is a linear map which is
both a morphism of $A$-modules and a morphism of $A$-comodules. The category
of right-right Yetter-Drinfeld $A$-modules will be denoted by 
$\mathsf{YD}(A)$. 
\end{definition}

\begin{remark}
If $A$ is a regular weak multiplier Hopf algebra in the sense of \cite{VDaWa}
(that is, it is a regular weak multiplier bialgebra with left and right full
comultiplication such that both $A$ and $A^\op$ possess an antipode) and $X$
is a vector space obeying the hypotheses of Theorem \ref{thm:YD_def}, then we
conclude by Remark \ref{rem:phihat_iso} that the equivalent assertions in
Theorem \ref{thm:YD_def} are equivalent also to the 
following. 
\begin{itemize}
\item[{(b')}] The datum $((X,\lambda,\varrho),\widehat\varphi_{X,-})$
is an object in the center of $\mathsf{M}^{(A)}$. 
\item[{(d')}] The datum $((X,\cdot),(\widehat\varphi_{-,X})^{-1})$ is an
object in the center of $\mathsf{M}_{(A)}$.
\end{itemize}
\end{remark}

\begin{theorem}\label{thm:J_functors}
For a regular weak multiplier bialgebra $A$ with right full comultiplication,
the following assertions hold.
\begin{itemize}
\item[{(1)}] There are fully faithful functors $J^{(A)}:\mathsf{YD}(A)\to
(\mathsf{M}_{(A)})^{\mathsf{M}_{(A)}}$ and $J_{(A)}:\mathsf{YD}(A)\to
{}^{\mathsf{M}^{(A)}}(\mathsf{M}^{(A)})$ rendering commutative the
following diagram (in which the unlabelled arrows in the top row denote the
evident forgetful functors). 
$$
\xymatrix{
\mathsf{M}^{(A)}\ar[d]_-{I^{(A)}}&&
\mathsf{YD}(A)\ar[ll]\ar[rr]\ar[ld]_-{J^{(A)}}\ar[rd]^(.6){J_{(A)}}&&
\mathsf{M}_{(A)} \ar[d]^-{I_{(A)}}\\
(\mathsf{M}_{(A)})^{U_{(A)}}&
(\mathsf{M}_{(A)})^{\mathsf{M}_{(A)}}\ar[l]^-{U_{(A)}}&&
{}^{\mathsf{M}^{(A)}}(\mathsf{M}^{(A)})\ar[r]_-{U^{(A)}}&
{}^{U^{(A)}}(\mathsf{M}^{(A)})}
$$
\item[{(2)}] If $I^{(A)}$ is an isomorphism (respectively, an equivalence)
 then so is $J^{(A)}$.
\item[{(3)}] If $I_{(A)}$ is an isomorphism (respectively, an equivalence)
 then so is $J_{(A)}$.
\end{itemize}
\end{theorem}

\begin{proof}
Recall that the morphisms $(V,\sigma)\to (V',\sigma')$ in the right weak
center of $\mathsf{M}_{(A)}$ are the $A$-module maps $V\to V'$ which are also
morphisms $(U_{(A)}V,U_{(A)}\sigma)\to (U_{(A)}V',U_{(A)}\sigma')$ in
$(\mathsf{M}_{(A)})^{U_A}$; and there is a similar description of the
morphisms in the left weak center of $\mathsf{M}^{(A)}$.

(1) First we show that there is a functor $J^{(A)}:\mathsf{YD}(A)\to
(\mathsf{M}_{(A)})^{\mathsf{M}_{(A)}}$, sending a Yetter-Drinfeld module
$(X,\cdot,\lambda,\varrho)$ to $((X,\cdot),\widehat\varphi_{-,X})$ and acting
on the morphisms as the identity map. By Theorem \ref{thm:YD_def}
(e)$\Rightarrow$(d), $((X,\cdot),\widehat\varphi_{-,X})$ is an object in the
right weak center of $\mathsf{M}_{(A)}$. A morphism in $\mathsf{YD}(A)$ is a
morphism of $A$-modules by definition and it is taken by $I^{(A)}$ --- acting
as the identity map --- to a morphism in $(\mathsf{M}_{(A)})^{U_{(A)}}$. Hence
it is a morphism in the right weak center of $\mathsf{M}_{(A)}$. Evidently,
$J^{(A)}$ is faithful and it renders commutative the left half of the
diagram. In order to see that it is full, take any morphism
$J^{(A)}(X,\cdot,\lambda,\varrho)\to J^{(A)}(X',\cdot,\lambda',\varrho')$ in
$(\mathsf{M}_{(A)})^{\mathsf{M}_{(A)}}$. It is by definition a morphism of
$A$-modules and a morphism in $(\mathsf{M}_{(A)})^{U_{(A)}}$. Thus by the
fullness of $I^{(A)}$ (see Proposition \ref{prop:I^{(A)}_ff}), it is a
morphism of $A$-comodules hence a morphism in $\mathsf{YD}(A)$.

The existence and the stated properties of $J_{(A)}$ are proven symmetrically. 
 
(2) Since $J^{(A)}$ is fully faithful by part (1), we only need to show that
if the object map of $I^{(A)}$ is bijective (resp. essentially surjective)
then so is the object map of $J^{(A)}$. To this end, pick up any object
$((V,\cdot),\sigma)$ in the right weak center of $\mathsf{M}_{(A)}$, taken by
the lifted functor $U_{(A)}$ to the object $(V,\sigma)$ in
$(\mathsf{M}_{(A)})^{U_{(A)}}$. By the assumption about $I^{(A)}$, there is a
unique (resp. some) object $(M,\lambda,\varrho)$ in $\mathsf{M}^{(A)}$ such
that $(M,\widehat \varphi_{-,M})$ is equal (resp. isomorphic) to $(V,\sigma)$ in
$(\mathsf{M}_{(A)})^{U_{(A)}}$. Then we can use the $R$-bimodule isomorphism
$M\cong V$ to induce an idempotent non-degenerate right $A$-action $\cdot$ on
$M$. By construction, the $R$-actions on $M$, corresponding to its
$A$-comodule structure and to its $A$-module structure, coincide. Hence
assertions (1.b) and (2.b) in Lemma \ref{lem:YD_actions} hold. Moreover,
since $\sigma_W$ is an $A$-module map for any object $W$ in
$\mathsf{M}_{(A)}$, so is $\widehat \varphi_{W,M}$. Thus by Theorem
\ref{thm:YD_def} (c)$\Rightarrow$(e), $(M,\cdot, \lambda,\varrho)$ is a
Yetter-Drinfeld module such that $J^{(A)}(M,\cdot, \lambda,\varrho)=
((M,\cdot), \widehat \varphi_{-,M})$ is equal (resp. isomorphic) to
$((V,\cdot),\sigma)$ in the right weak center of $\mathsf{M}_{(A)}$. 
 
Part (3) is proven symmetrically.
\end{proof}

\section{The monoidal category of Yetter-Drinfeld modules}\label{sec:YD_cat}

The aim of this section is to prove that the category of Yetter-Drinfeld
modules, over a regular weak multiplier bialgebra $A$ with right full
comultiplication, carries a monoidal structure admitting strict monoidal
forgetful functors to the category of idempotent non-degenerate $A$-modules
and the category of full $A$-comodules. If $A$ is in addition a regular weak
multiplier Hopf algebra in the sense of \cite{VDaWa} (that is, the
comultiplication is both left and right full and both $A$ and $A^\op$ possess
an antipode), then finite dimensional Yetter-Drinfeld modules are shown to
possess duals in this category.
 
\begin{proposition}\label{prop:prod_YD}
Consider a regular weak multiplier bialgebra $A$ with right full
comultiplication. For any right-right Yetter-Drinfeld $A$-modules $X$ and $Y$,
the module tensor product $X\ox_R Y$ over the base algebra $R$ of $A$ is again
a right-right Yetter-Drinfeld $A$-module via the module structure in
\cite[Proposition 5.5]{BoGTLC:wmba} and the comodule structure in
\cite[Proposition 5.4]{Bo:wmba_comod}.
\end{proposition}

\begin{proof}
By \cite[Proposition 5.5]{BoGTLC:wmba}, $X\ox_R Y$ is an idempotent
non-degenerate right $A$-module. By \cite[Proposition 5.4 and Proposition
5.6]{Bo:wmba_comod}, $X\ox_R Y$ is a full right $A$-comodule. So it remains to
check the compatibility between them. Since the $R$-actions corresponding to
the $A$-module structure, and the $R$-actions corresponding to the
$A$-comodule structure, coincide both on $X$ and $Y$, they evidently coincide
on $X\ox_R Y$ too. Since by Theorem \ref{thm:YD_def}~(e)$\Rightarrow$(c) both
$\widehat\varphi_{V,X}$ and $\widehat\varphi_{V,Y}$ are morphisms of
$A$-modules, for any object $V$ in $\mathsf{M}_{(A)}$, so is
$\widehat\varphi_{V,X\ox_R Y}$ by Lemma \ref{lem:phi_hat}~(4). Thus $X\ox_R Y$
is a Yetter-Drinfeld module by Theorem \ref{thm:YD_def}~(c)$\Rightarrow$(e). 
\end{proof}

\begin{proposition}\label{prop:R_YD}
Consider a regular weak multiplier bialgebra $A$ with right full
comultiplication. Then the base algebra $R$ is a right-right Yetter-Drinfeld
module via the action $\sqcap^R(a)\cdot b=\sqcap^R(\sqcap^R(a)b)=\sqcap^R(ab)$
(see \cite[Proposition 5.3]{BoGTLC:wmba}) and the comodule structure
$(\lambda:r\ox a\mapsto E(1\ox ra), \varrho:r\ox a \mapsto (1\ox ar)E)$ (see
\cite[Example 2.5 and Example 4.4]{Bo:wmba_comod}).
\end{proposition}

\begin{proof}
By \cite[Proposition 5.3]{BoGTLC:wmba}, $R$ is an idempotent non-degenerate
right $A$-module via the stated action. By \cite[Example 2.5 and Example
4.4]{Bo:wmba_comod} $R$ is a full right $A$-comodule via the stated maps. So
we need to check the compatibility of these structures. Both the $R$-actions
corresponding to the $A$-module structure, and the $R$-actions corresponding
to the $A$-comodule structure, are given by the multiplication in $R$, see
\cite[Proposition 5.3]{BoGTLC:wmba} and \cite[Example 4.8]{Bo:wmba_comod},
respectively. Hence they are equal so that assertions (1.b) and (2.b) in Lemma
\ref{lem:YD_actions} hold. By Lemma \ref{lem:phi_hat}~(1),
$((R,\lambda,\varrho), \widehat\varphi_{R,-})$ is the monoidal unit of
${}^{\mathsf{M}^{(A)}} (\mathsf{M}^{(A)})$ and thus $R$ is a Yetter-Drinfeld 
module by Theorem \ref{thm:YD_def} (b)$\Rightarrow$(e). 
\end{proof}

In view of \cite[Theorem 5.6]{BoGTLC:wmba}, idempotent non-degenerate right
$A$-modules constitute a monoidal category via the $R$-module tensor product
as a monoidal product. By \cite[Theorem 5.7]{Bo:wmba_comod}, also full right
$A$-comodules constitute a monoidal category via the $R$-module tensor product
as a monoidal product. Hence Proposition \ref{prop:prod_YD} and Proposition
\ref{prop:R_YD} give rise to the following.

\begin{theorem} \label{thm:YD_moncat}
For a regular weak multiplier bialgebra $A$ with right full comultiplication,
the category $\mathsf{YD}(A)$ of right-right Yetter-Drinfeld $A$-modules is
monoidal such that there is a commutative diagram of strict monoidal forgetful
functors 
$$
\xymatrix{
\mathsf{YD}(A)\ar[r]\ar[d]&\mathsf{M}_{(A)}\ar[d]\\
\mathsf{M}^{(A)}\ar[r]&{}_R\mathsf{M}_R .}
$$
\end{theorem}

If a usual, unital weak bialgebra as in \cite{WHAI,Nill} possesses a bijective
antipode, then it was shown in \cite[Section 5]{Cae:YD} that finite
dimensional objects posses duals in the category of Yetter-Drinfeld
modules. Below we discuss the duality of finite dimensional Yetter-Drinfeld
modules over a regular weak multiplier Hopf algebra in the terminology of
\cite{VDaWa} --- that is, over a regular weak multiplier bialgebra $A$ with
left and right full comultiplication such that both $A$ and its opposite
possess an antipode.

\begin{theorem}\label{thm:dual}
Let $A$ be a regular weak multiplier Hopf algebra (in the sense of
\cite{VDaWa}) over a field $k$. Then any finite dimensional right-right
Yetter-Drinfeld module possesses a dual in $\mathsf{YD}(A)$.
\end{theorem}

\begin{proof}
Recall from \cite[Proposition 4.3]{VDaWa} that the antipode on $A$
restricts to an (anti-algebra) isomorphism $S:A\to A$. We denote its inverse
by $S^{-1}$.

Let $(X,\cdot,\lambda,\varrho)$ be a finite dimensional right-right
Yetter-Drinfeld $A$-module. By the anti-multiplicativity of $S$
(cf. \cite[Proposition 3.5]{VDaWa}), the linear dual $X^*:=\mathsf{Lin}(X,k)$
is a right $A$-module via the action
$$
(\varphi\cdot a)(x):=\varphi(x\cdot S^{-1}(a)),\qquad
\forall \varphi\in X^*,\ x\in X,\ a\in A.
$$
Let us see that this $A$-action is surjective. Since the algebra $A$ has local
units by \cite[Proposition 4.9]{VDaWa}, since the $A$-action on $X$ is
surjective and since $X$ is finite dimensional, there is an element $p\in A$
such that $x\cdot p =x$ for all $x\in X$. For this element $p$, any
$\varphi\in X^*$ and $x\in X$,
$$
(\varphi\cdot S(p))(x)=
\varphi(x\cdot p)=
\varphi(x)
$$
so that $\varphi=\varphi\cdot S(p)$ and thus the $A$-action on $X^*$ is
surjective. The $A$-action on $X^*$ is also non-degenerate. Indeed, if
$\varphi\cdot a=0$ for all $a\in A$, then
$$
0=(\varphi\cdot a)(x)=\varphi(x\cdot S^{-1}(a))\qquad
\forall \ x\in X,\ a\in A,
$$
so by the bijectivity of $S^{-1}$ and surjectivity of the $A$-action on $X$, 
$\varphi(x)=0$ for all $x\in X$ proving that $\varphi=0$.

Via the maps $\lambda^{*S},\varrho^{*S}:X^*\ox A\to X^*\ox A$ in (6.11) of
\cite{Bo:wmba_comod}, $X^*$ is a full $A$-comodule, see \cite[Proposition 6.2
and Proposition 6.5]{Bo:wmba_comod}. So it remains to check the
Yetter-Drinfeld compatibility conditions. The $R$-actions on the above
$A$-module $X^*$ are given by 
\begin{eqnarray*}
(\varphi\cdot a\wr \sqcap^R(b))(x\cdot c)&=&
(\varphi\cdot a \sqcap^R(b))(x\cdot c)=
\varphi(x\cdot cS^{-1}(a\sqcap^R(b)))\\
&=&
\varphi(x\cdot c\overline \sqcap^L(b) S^{-1}(a))=
(\varphi\cdot a)(\sqcap^R(b)\wl (x\cdot c))\\
(\sqcap^R(b)\wl \varphi\cdot a)(x\cdot c)&=&
(\varphi\cdot a \overline \sqcap^L(b))(x\cdot c)=
\varphi(x\cdot cS^{-1}(a\overline\sqcap^L(b)))\\
&=&
\varphi(x\cdot c(\vartheta^{-1}\sqcap^R(b)) S^{-1}(a))=
(\varphi\cdot a)((x\cdot c)\wr (\vartheta^{-1}\sqcap^R(b))),
\end{eqnarray*}
for any $\varphi\in X^*$, $x\in X$ and $a,b,c\in A$, where $\vartheta$ is the
Nakayama automorphism of the firm Frobenius algebra $R$, cf. \cite[Theorem
4.6~(3)]{BoGTLC:wmba}. 
In the third equality of the first computation we used \cite[Lemma
6.14]{BoGTLC:wmba} and in the third equality of the second computation we used
the same lemma together with \cite[Proposition 4.9]{BoGTLC:wmba}. 
Comparing this with \cite[Proposition 6.6]{Bo:wmba_comod} and using that the
$R$-actions \eqref{eq:w_actions} and \eqref{eq:b_actions} on $X$ coincide, we
conclude that they coincide on $X^*$ too. For any $b,d\in A$, introduce the
index notation $T_1(b\ox d)=:b^1\ox d^1$; for $x\in X$ and $d\in A$ introduce
the index notation $\lambda(x\ox d)=:x^\lambda\ox d^\lambda$ and $\varrho(x\ox
d)=:x^\varrho\ox d^\varrho$; where in all cases implicit summation is
understood. For any $\varphi\in X^*$ and $a,c\in A$,
\begin{eqnarray*}
\varrho^{*S}[(\varphi\ox cS(d))&&\hspace{-1cm}(\Delta^{\mathsf{op}}S(b))]
(1\ox S(a))\\
&=&(\varrho^{*S}(\varphi\cdot(-)\ox c(-))(S\ox S)T_1(b\ox d))(1\ox S(a))\\
&=&\varphi((-)^\lambda\cdot b^1) \ox cS(ad^{1\lambda})
=\varphi((-)^\varrho\cdot b^1) \ox cS(a^\varrho d^1),
\end{eqnarray*}
where the last equality follows by the compatibility condition (2.10) in
\cite{Bo:wmba_comod}. Denoting also $T_2(a\ox b)=:a^2\ox b^2$ (with implicit
summation understood),
\begin{eqnarray*}
\varrho^{*S}(\varphi\ox cS(d))(\Delta S(b))(1\ox S(a))
&=&\varrho^{*S}(\varphi\ox cS(d))((S\ox S)T_2^{21}(b\ox a))\\
&=&\varphi((-)^\lambda)\cdot S(b^2)\ox cS(a^2d^\lambda)\\
&=&\varphi(((-)\cdot b^2)^\lambda)\ox cS(a^2d^\lambda)\\ 
&=&\varphi(((-)\cdot b^2)^\varrho)\ox cS(a^{2\varrho}d)
\end{eqnarray*}
where the last equality follows again by (2.10) in \cite{Bo:wmba_comod}. These
expressions are equal since applying the Yetter-Drinfeld compatibility
condition in the second equality,
$$
x^\varrho\cdot b^1\ox a^\varrho d^1=
\varrho(x\ox a)\Delta(b)(1\ox d)=
\varrho[(x\ox a)\Delta^{\mathsf{op}}(b)](1\ox d)=
(x\cdot b^2)^\varrho \ox a^{2\varrho}d
$$
for any $x\in X$. By the bijectivity of $S$ and the non-degeneracy of $A$,
this proves that $X^*$ is a Yetter-Drinfeld module.

For the comultiplication
$$
\delta:R\to R\ox R,\qquad r\mapsto (r\ox 1)F=F(1\ox r)
$$
in \cite[Proposition 4.3~(3)]{BoGTLC:wmba}, introduce the implicit summation
index notation $\delta(r)=:r_1\ox r_2$. The evaluation map 
$$
\mathsf{ev}:X^*\ox_R X \to R,\qquad
\varphi\ox_R x\br r \mapsto \varphi(x\br r_1)r_2
$$ 
in (6.13) in \cite{Bo:wmba_comod} and the coevaluation map 
$$
\mathsf{coev}:R\to X\ox_R X^*,\qquad 
r\mapsto \sum_i r\bl x_i \ox_R \xi^i
$$ 
in (6.14) in \cite{Bo:wmba_comod} --- where $\{x_i\in X\}$ and $\{\xi^i\in
X^*\}$ are finite dual bases --- obey the triangular identities of duality and
they are morphisms of $A$-comodules by \cite[Theorem 6.7]{Bo:wmba_comod}. So
in order to complete the proof, we need to see that they are morphisms of
$A$-modules too. For any $\varphi\in X^*$, $x\in X$, $a,b\in A$ and $r,s\in
R$,
\begin{eqnarray*} 
\mathsf{ev}[(\varphi \ox_R x\cdot ar)\cdot bs]&=&
\mathsf{ev}\,\pi_{X^*,X}(\varphi\cdot(-) \ox x\cdot(-)) T_3(bs\ox ar)\\
&=&
\varphi(x\cdot \mu^{\mathsf{op}}(S^{-1}\ox A)T_3(bs_1\ox ar))s_2\\
&=&
\varphi(x\cdot ar\overline \sqcap^R(bs_1))s_2
=\varphi(x\cdot a\overline \sqcap^R(rbs_1))s_2\\
&=&
\sqcap^R(\varphi(x\cdot ar_1)r_2bs)
=\mathsf{ev}(\varphi\ox x\cdot ar)\cdot bs,
\end{eqnarray*}
so that the evaluation map is a morphism of $A$-modules.
In the second equality we used twice that by \cite[Lemma 3.3]{BoGTLC:wmba},
$T_3(bs\ox a)=T_3(b\ox a)(1\ox s)$. In the third equality we applied an
identity in (6.14) in \cite{BoGTLC:wmba} to the opposite weak multiplier Hopf
algebra $A^{\mathsf{op}}$. The fourth equality follows by \cite[Lemma
3.4]{BoGTLC:wmba}. In the penultimate equality we used that by Lemma
\ref{lem:dual_YD}~(b)=(d), 
$$
\overline \sqcap^R(rbs_1) \ox s_2=
(\overline \sqcap^R \ox R)[(rbs\ox 1)F]=
(R\ox \sqcap^R)[F(1\ox rbs)]=
r_1\ox \sqcap^R(r_2bs).
$$
Recall from (6.12) in \cite{Bo:wmba_comod} the vector space isomorphism 
$$
\kappa:X\ox_R X^*\to \mathsf{Hom}_R(X,X),\qquad 
x\ox_R \varphi \mapsto[y\br r\mapsto \varphi(y\br r_1)x\br r_2].
$$
For any $y\in X$, $r\in R$ and $a,b,c\in A$, omitting for brevity the
summation symbol over $i$, 
\begin{eqnarray*}
[\kappa((\mathsf{coev}\sqcap^R(b))\cdot a)(y\br r)]\cdot c&=&
[(x_i\cdot\overline\sqcap^L(b)(-)\ox \xi^i\cdot(-))(y\br r_1)]
T_4(r_2 c \ox a)\\
&=&
[x_i\cdot\overline\sqcap^L(b)(-)\ox \xi^i(y\cdot r_1S^{-1}(-))]
T_4(r_2 c \ox a)\\
&=&
[x_i\cdot(-) \ox \xi^i(y\cdot r_1S^{-1}(-))]
T_4(r_2 c \ox \overline\sqcap^L(b)a)\\
&=&
y\cdot r_1(\mu^{\mathsf{op}}(A\ox S^{-1})T_4(r_2c\ox \overline\sqcap^L(b)a))\\
&=&
y\cdot r_1\overline\sqcap^L(\overline\sqcap^L(b)a)r_2 c=
y\cdot r\overline\sqcap^L(ba)c\\
&=&
[(\kappa\,\mathsf{coev}(\sqcap^R(b)\cdot a))(y\br r)]\cdot c.
\end{eqnarray*}
Hence by the non-degeneracy of the right $A$-module $X$, since $\kappa$ is
an isomorphism and by the surjectivity of the right $R$-action on $X$, it
follows that the coevaluation map is a morphism of $A$-modules. In the third
equality we used that by \cite[Lemma 3.3]{BoGTLC:wmba},
$(\overline\sqcap^L(b)\ox 1)T_4(c\ox a)=T_4(c\ox \overline\sqcap^L(b)a)$, for
any $a,b,c\in A$. In the fifth equality we applied an identity in (6.14) in
\cite{BoGTLC:wmba} to the opposite weak multiplier Hopf algebra
$A^{\mathsf{op}}$. The penultimate equality follows by \cite[Lemma
3.2]{BoGTLC:wmba}, \cite[Lemma 3.5]{BoGTLC:wmba}, and the fact that
$\mu\delta(r)=r_1 r_2=r$, see \cite[Proposition 4.3~(1)]{BoGTLC:wmba}. 
\end{proof}

\end{document}